\documentclass[a4paper]{article}
\usepackage{amssymb, amsthm, amsmath}
\usepackage{graphicx}
\usepackage{color}
\usepackage{dsfont}
\usepackage{bm}
\usepackage{pgf,tikz}
\usepackage{hyperref}
\usepackage{authblk}
\usepackage{xcolor}
\usetikzlibrary{arrows}

\newtheorem{theorem}{Theorem}[section]
\newtheorem{proposition}[theorem]{Proposition}
\newtheorem{lemma}[theorem]{Lemma}

\theoremstyle{definition}
\newtheorem{definition}[theorem]{Definition}
\newtheorem{assumption}[theorem]{Assumption}
\newtheorem{remark}[theorem]{Remark}
\newtheorem{example}[theorem]{Example}

\renewcommand{\Game}{\mathfrak{G}}

\newcommand{\wt }{\widetilde }
\newcommand{\wh }{\widehat }

\newcommand{\R}{\mathbb{R}}

\newcommand{\E}{\mathbb{E}}

\renewcommand{\P}{\mathbb{P}}

\newcommand{\FF}{\mathbb{F}}

\newcommand{\EP}{ {\mathbb E}_{\mathbb{P}}}

\newcommand{\Filt}{\mathcal{F}}

\newcommand{\Strat}{\mathcal{S}}
\newcommand{\Tstrat}{\mathcal{T}}
\newcommand{\STOP}{\mathcal{T}}

\newcommand{\esssup}{\operatornamewithlimits{ess\,sup}}
\newcommand{\essinf}{\operatornamewithlimits{ess\,inf}}

\let\inf\relax \DeclareMathOperator*\inf{\vphantom{p}inf}
\let\essinf\relax \DeclareMathOperator*\essinf{\vphantom{p}ess\,inf}

\newcommand\I{\mathds{1}}

\newcommand{\si}{\sigma}
\newcommand{\ep}{\epsilon}

\newcommand{\htau}{\wh \tau}

\newcommand{\hsi}{\wh \si}
\newcommand{\tsi}{\wt \si}

\newcommand{\DG}{\operatorname{DG}}   
\newcommand{\SDG}{\operatorname{SDG}} 
\newcommand{\GDG}{\operatorname{GDG}} 

\title{On Dynkin Games with Unordered Payoff Processes}
\author[1,2]{Ivan Guo}

\affil[1]{School of Mathematics, Clayton Campus, Monash University, VIC, 3800, Australia}
\affil[2]{Centre for Quantitative Finance and Investment Strategies\thanks{\textbf{Acknowledgements\quad} The Centre for Quantitative Finance and Investment Strategies has been supported by BNP Paribas.
I. Guo has been partially supported by the Australian Research Council (Grant DP170101227).
}, Monash University, Australia}

\begin{document}
\maketitle

\begin{abstract}
A {Dynkin game} is a zero-sum, stochastic stopping game between two players where either player can stop the game at any time for an observable payoff.
Typically the payoff process of the max-player is assumed to be smaller than the payoff process of the min-player, while the payoff process for simultaneous stopping is in between the two.
In this paper, we study general Dynkin games whose payoff processes are in arbitrary positions. In both discrete and continuous time settings, we provide necessary and sufficient conditions for the existence of pure strategy Nash equilibria and $\ep$-optimal stopping times in all possible subgames.


\end{abstract}

\section{Introduction}\label{sec2a0}

A \emph{Dynkin game}, first introduced by Dynkin \cite{dynkin1969game}, is a zero-sum, stochastic stopping game between two players where either player can stop the game at any time for an observable payoff. Much research has been done in this field as well as various related problems, for example, \cite{cvitanic1996backward, ekstrom2008optimal, hamadene2010continuous, laraki2005value, peskir2009optimal, rosenberg2001stopping, solan2001quitting, solan2003deterministic, touzi2002continuous}. One interesting application of Dynkin games is in two-person game contingent claims. The two-person game contingent claim is defined by Kifer \cite{kifer2000game}, who also proved the existence and uniqueness of its arbitrage price.
Further works, such as Hamad\`{e}ne and Zhang \cite{hamadene2010continuous}
and Kallsen and K\"uhn \cite{kallsen2004pricing}, studied various techniques in its pricing.

Typically the Dynkin game is associated with the payoff processes $X, Y$ and $Z$. In particular, the payoff is given by $X$ if the max-player stops first, $Y$ if the min-player stops first, and $Z$ if both players stop at the same time. \emph{Standard Dynkin games}, commonly studied in literature, refer to cases where the inequality $X\leq Z\leq Y$ is satisfied. This chapter will present some new results for \emph{general Dynkin games}, whose payoff processes are in arbitrary positions.

Sections \ref{sec3.11} and \ref{sec3.21} examines the standard Dynkin game in a discrete-time set-up. Well-known results addressing the existence and uniqueness of value as well as optimal stopping times are presented in Propositions \ref{pro:22} and \ref{propep05}. In Sections \ref{sec3.12} and \ref{sec3.22}, we establish some original results for the general Dynkin game in both discrete and continuous-time settings. In particular, the main results are Theorems \ref{coreo25} and \ref{thmez10}, which provide sufficient conditions for the existence and uniqueness of value and optimal stopping times. The same conditions are then shown to be necessary for the existence of value in all possible subgames.

The theory of two-person non-zero-sum Dynkin games is not included here. We instead refer the reader to
Hamad{\`e}ne and Zhang \cite{hamadene2010continuous}, Hamad{\`e}ne and Hassani \cite{hamadene2012multi},
   Ohtsubo \cite{ohtsubo1987nonzero,ohtsubo1991discrete},
  and Shmaya and Solan \cite{shmaya2004two} for some partial results in this area.

%

\section{Discrete-Time Dynkin Games}\label{sec2a}

We first present in Section \ref{sec3.11} the classic results on discrete-time zero-sum Dynkin games. Subsequently, in Section  \ref{sec3.12}, we attempt to provide a complete solution to the problem of existence of a Nash equilibrium for the general zero-sum Dynkin game. It should be stressed that we only deal
with stopping games with a finite time horizon; a large body of the existing literature is devoted to stopping games with infinite time
 horizon and thus also with possibly infinite optimal stopping times.

We will first examine zero-sum stopping games with the random payoff given by
\begin{gather} \label{zt44tt}
R(\tau , \si ) = \I_{\{ \tau < \si \}}\, X_{\tau }  +  \I_{\{ \si  < \tau \}}\, Y_{\si } +  \I_{\{ \si  = \tau \}}\, Z_{\si },
\end{gather}
where $X,Y$ and $Z$ are $\FF$-adapted and integrable processes.
 The random times $\tau$ and $\si$ are chosen from the class $\Tstrat_{[0,T]}$
 of $\FF$-stopping times and they are interpreted as the respective stopping strategies of the two players.

\begin{remark}\label{remec00}
By assumption, $\tau, \si \leq T$ and thus the values of $X_T$ and $Y_T$ are irrelevant in what follows.
Therefore, without loss of generality, we adopt the common convention that $X_T=Z_T=Y_T$.
\end{remark}

The following definition deals with the discrete-time  case, but its extension to the continuous-time
framework is immediate.

\begin{definition} \label{def:dynkin}
For any fixed date $t=0,1, \dots ,T$,  by the \emph{Dynkin game} $\DG_t(X,Y,Z)$ started at time $t$ and associated with the payoff $R (\tau , \si)$, we mean a zero-sum two-person stochastic game in which the goal of the {\it
max-player}, who controls a stopping time $\tau_t \in \STOP_{[t,T]}$, is to maximise the conditional expectation
\begin{gather}
\label{tt44t} \E_\P ( R(\tau_t ,\si_t )\,|\, \Filt_{t}),
\end{gather}
while the {\it min-player},  controlling a stopping time $\si_t \in \STOP_{[t,T]}$, wishes to minimise the conditional expectation
(\ref{tt44t}). Also denote by $\DG (X,Y,Z)$ the family of Dynkin games associated with $R(\tau,\si)$.
\end{definition}

 For any fixed $t$ and arbitrary stopping times $\tau_t $ and $\si_t $ from the class $\STOP_{[t,T]}$,
 formula (\ref{zt44tt}) yields
\begin{gather}  \label{tt44tt}
\E_\P ( R(\tau_t ,\si_t )\,|\, \Filt_{t})= \E_\P \Big(\sum_{u=t}^T \big( \I_{\{ u=\tau_t < \si_t \}}
\, X_u +  \I_{\{ u=\si_t< \tau_t \}}\, Y_u +  \I_{\{ u =\si_t = \tau_t \}}\, Z_u \big) \,\Big|\, \Filt_{t}\Big).
\end{gather}
We are interested in finding the {\it value process} $V^*$ of $\DG (X,Y,Z)$, that is, an $\FF$-adapted
process such that, for all $t=0,1,\dots ,T$,
\begin{gather*}
V^*_t = \essinf_{\si_t \in\STOP_{[t,T]}}\esssup_{\tau_t \in\STOP_{[t,T]}} \E_\P
(R(\tau_t,\si_t)\,|\,\Filt_{t})=\esssup_{\tau_t \in\STOP_{[t,T]}}\essinf_{\si_t \in\STOP_{[t,T]}} \E_\P
(R(\tau_t ,\si_t)\,|\,\Filt_{t}).
\end{gather*}
In addition, we search for a corresponding Nash (hence also optimal) equilibrium, that is, any pair $(\tau_t^*, \si_t^*)$ of
{\it optimal stopping times} satisfying
\begin{gather*}
V^*_t = \E_\P(R(\tau_t^*,\si_t^*)\,|\,\Filt_{t}).
\end{gather*}

\subsection{Standard Dynkin Game} \label{sec3.11}

We first present well-known results for the special class of two-person, zero-sum stopping games in the discrete-time
framework (see Neveu \cite{neveu1975discrete}).

\begin{definition}
By the {\it standard Dynkin game} $\SDG (X,Y,Z)$, we mean the stochastic stopping game associated with the payoff
$R$ given by \eqref{zt44tt} with processes $X,Y$ and $Z$ satisfying the following condition: $X \leq Z \leq Y$.
\end{definition}

The following definitions introduces candidates for the value process of the standard zero-sum Dynkin game and the optimal stopping times.

\begin{definition} \label{def:vp1}
The process $V$ is defined by setting $V_T = Z_T$ and, for any $t=0,1,\dots ,T-1$,
\begin{gather}
V_{t} = \min \Big\{ Y_t,\, \max  \big\{ X_{t}, \E_\P  ( V_{t+1}\,|\,\Filt_{t}) \big\} \Big\}= \max \Big\{ X_{t},\, \min \big\{ Y_t, \E_\P  ( V_{t+1}\,|\,\Filt_{t}) \big\} \Big\}.\label{lemec101}
\end{gather}
Furthermore, we set, for any fixed $t=0,1, \dots ,T$,
\begin{gather} \label{sig}
\tau^{*}_{t} := \min \big\{ u \in \{ t, t+1, \dots , T \} \,|\, V_u= X_u \big\},\\
\label{tau}
\si^{*}_{t} :=  \min \big\{ u \in \{ t, t+1, \dots , T \} \,|\, V_u = Y_u \big\}.
\end{gather}
\end{definition}

The assumption that $X\leq Z\leq Y$ immediately implies that the second equality in \eqref{lemec101} holds and, for $t=0,1,\ldots,T$,
\begin{gather}
X_t \leq V_t \leq Y_t ,\label{lemec102}
\end{gather}
 so that the process $V$ is bounded below $X$ and above by $Y$. The stopping times $\tau^*_t$ and $\si^*_t$ capture the first moment $V$ hits the lower and upper boundaries, respectively, starting from time $t$. Obviously, if $V$ is the value process then we also must have, for $t=0,1,\ldots,T$,
\begin{gather*} 
V_t=\E_\P \big(R(\tau^{*}_{t},\si^{*}_{t} )\,|\,  \Filt_{t}\big).
\end{gather*}

%
%

The following classic result shows that the process $V$ given by \eqref{lemec101} is indeed equal to the
value process $V^*$ of $\SDG (X,Y,Z)$. Recall that we work here under the standing
assumption that $X \leq Z \leq Y$; this condition will be relaxed in the foregoing subsection.

\begin{proposition} \label{pro:22}
(i) Let the process $V$ and the stopping times $\tau^{*}_{t},\si^{*}_{t}$ be given by Definition \ref{def:vp1}.
Then we have, for arbitrary stopping times $\tau_t , \si_t \in\STOP_{[t,T]}$,
\begin{gather} \label{eqe11}
\E_\P \big( R(\tau^{*}_{t},\si_t )\,|\, \Filt_{t}\big) \geq V_{t}
\geq \E_\P \big( R(\tau_t ,\si^{*}_{t})\,|\, \Filt_{t}\big),
\end{gather}
and thus also
\begin{gather*} 
\E_\P \big( R(\tau^{*}_{t},\si_t )\,|\,  \Filt_{t}\big) \geq \E_\P \big(
R(\tau^{*}_{t},\si^{*}_t )\,|\,  \Filt_{t}\big) \geq \E_\P \big(
R(\tau_t ,\si^{*}_{t})\,|\,  \Filt_{t}\big).
\end{gather*}
Hence $(\tau^{*}_{t},\si^{*}_t )$ is a Nash equilibrium of the standard Dynkin game $\SDG_t(X,Y,Z)$.

\noindent (ii) The process $V$ is the value process
of the game $\SDG (X,Y,Z)$, that is, for every $t=0,1,\dots ,T$,
\begin{align*} 
V_t &= \essinf_{\si_t \in\STOP_{[t,T]}} \,
\esssup_{\tau_t \in\STOP_{[t,T]}} \E_\P \big(R(\tau_t ,\si_t )\,|\,
\Filt_{t}\big)=\EP \big(R(\tau^{*}_t ,\si^{*}_t)\,|\, \Filt_{t}\big)
 \\ &=  \esssup_{\tau_t \in\STOP_{[t,T]}} \, \essinf_{\si_t \in\STOP_{[t,T]}}
\E_\P \big(R(\tau_t ,\si_t )\,|\, \Filt_{t}\big) = V_t^* , \nonumber
\end{align*}
and thus $\tau^{*}_t$ and $\si^{*}_t$ are optimal stopping times as of time $t$. In particular, $V^*_T = Z_T$
and for any $t=0,1,\dots ,T-1$,
\begin{gather}  \label{eq:crcr}
V_{t}^* = \min \Big \{ Y_t,\, \max \big\{ X_{t}, \E_\P  ( V_{t+1}^*\,|\,\Filt_{t}) \big\} \Big\}.
\end{gather}
\end{proposition}

\begin{proof} (i) We apply the backward induction. The inequalities \eqref{eqe11} clearly hold for $t=T$. Assume that \eqref{eqe11} holds for some $t$, that is, for arbitrary  $\tau_{t},\si_{t}\in T_{[t,T]}$,
\begin{gather}\label{eqed09}
\E_\P \big( R(\tau^{*}_{t},\si_{t} )\,|\,  \Filt_{t}\big) \geq V_{t} \geq \E_\P \big(
R(\tau_{t} ,\si^{*}_{t})\,|\,  \Filt_{t}\big).
\end{gather}
We wish to prove that, for arbitrary  $\tau_{t-1},\si_{t-1}\in T_{[t-1,T]}$,
\begin{gather}\label{eqed10}
\E_\P \big( R(\tau^{*}_{t-1},\si_{t-1} )\,|\,  \Filt_{t-1}\big) \geq V_{t-1} \geq \E_\P \big(
R(\tau_{t-1} ,\si^{*}_{t-1})\,|\,  \Filt_{t-1}\big).
\end{gather}
There are essentially two cases, which are dealt with using different arguments.
\begin{itemize}
\item First, if the game is stopped at time $t-1$, then the result can be deduced by analysing the relative sizes of processes $X, Y, Z$ and $V$ at time $t$.
\item Second, if the game is not stopped at time $t-1$, then the analysis is reduced to the time $t$ case, which is covered by the induction hypothesis.
\end{itemize}
Note that since the game is symmetric between the two players, it suffices to establish the upper inequality of \eqref{eqed10}. The lower inequality can be shown using analogous arguments.

For any $\tau_{t-1},\si_{t-1}\in T_{[t-1,T]}$, let us write $\tilde \tau_{t-1}:=\tau_{t-1}\vee t$, $\tilde \si_{t-1}:=\si_{t-1}\vee t$, so that the stopping times $\tilde \tau_{t-1}$ and $\tilde \si_{t-1}$ belong to $\STOP_{[t,T]}$.

We proceed to the proof of the upper inequality in \eqref{eqed10}, beginning with the case where the game is stopped at time $t-1$. On the event $\{\tau^*_{t-1}=t-1\}$,
\begin{gather}\label{eqed11}
 \E_\P \big(R(\tau^{*}_{t-1},\si_{t-1} )\,|\,  \Filt_{t-1}\big)  =
  \I_{\{\si_{t-1} > t \}}X_{t-1} + \I_{\{ \si_{t-1} = t \}} Z_{t-1}
\geq X_{t-1} = V_{t-1}.
\end{gather}
On the event $\{\si_{t-1}=t-1 < \tau^*_{t-1}\}$, using \eqref{lemec102}, we obtain
\begin{align}\label{eqed13}
\E_\P \big(R(\tau^{*}_{t-1},\si_{t-1} )\,|\,  \Filt_{t-1}\big)  = Y_{t-1} \geq V_{t-1}.
\end{align}
Now for the case where the game is not stopped at time $t-1$. On the event $\{\tau^*_{t-1}\wedge \si_{t-1}\geq t\}$, it follows from Definition \ref{def:vp1} that $\tau^*_{t-1}=\tilde \tau^*_{t-1}=\tau^*_{t}$ and $V_{t-1}>X_{t-1}$, and thus
\begin{gather}\label{eqed12}
V_{t-1}=\min\big\{Y_{t-1},\E_\P \big( V_{t} \,|\,  \Filt_{t-1}\big)\big\}.
\end{gather}
Hence
\begin{align}
\E_\P \big(R(\tau^{*}_{t-1},\si_{t-1} )\,|\,  \Filt_{t-1}\big)  &= \E_\P \big( R(\tilde \tau^{*}_{t-1},\tilde \si_{t-1} )\,|\,  \Filt_{t-1}\big) \nonumber\\
&= \E_\P \big( \E_\P(R(\tau^{*}_{t},\tilde \si_{t-1} )\,|\,  \Filt_t\big) \,|\,  \Filt_{t-1}\big)\nonumber\\
&\geq \E_\P \big( V_{t} \,|\,  \Filt_{t-1}\big) \label{eqed15}\\
&\geq \min\big\{Y_{t-1},\E_\P \big( V_{t} \,|\,  \Filt_{t-1}\big)\big\} \nonumber\\
&=V_{t-1}.\label{eqed16}
\end{align}
Note that inequality \eqref{eqed15} follows from the induction hypothesis \eqref{eqed09}, while equality \eqref{eqed16} is an immediate consequence of \eqref{eqed12}. After combining \eqref{eqed11}, \eqref{eqed13} and \eqref{eqed16}, we obtain
the upper inequality of \eqref{eqed10}. As mentioned before, the lower inequality can be \eqref{eqed10} established by symmetry. The induction is then complete and thus statement (i) is proven.

\noindent (ii) We observe that, from \eqref{eqe11}, the pair $(\tau^*_t,\si^*_t)$ is a Nash equilibrium of $\SDG_t(X,Y,Z)$. Therefore, the process $V$ satisfies $V_t=\E_\P \big(R(\tau^{*}_{t},\si^{*}_{t} )\,|\,  \Filt_{t}\big)$ and thus it is the value process of the
standard zero-sum Dynkin game $\SDG (X,Y,Z)$. Equality \ref{eq:crcr} now follows easily.
\end{proof}

\begin{remark} \label{remec30}
It can be easily checked from Definition \ref{def:vp1} that the stopped process $V^{\tau^{*}_t \wedge \si^{*}_t}$ is an $\FF$-martingale on the time interval $[t,T]$.
\end{remark}

\subsection{General Dynkin Game}  \label{sec3.12}

We will now discuss possible generalisations of the standard zero-sum Dynkin game, while still maintaining the zero-sum property of the game. Specifically, we consider the zero-sum Dynkin game  associated with the random payoff $R$ given by
\begin{gather}\label{eqee00}
R(\tau , \si ) = \I_{\{ \tau < \si \}}\, X_{\tau }  +  \I_{\{ \si  < \tau \}}\,
Y_{\si }  +  \I_{\{ \si  = \tau \}}\, Z_{\si },
\end{gather}
where $X, Y$ and $Z$ are $\FF$-adapted, integrable processes. Note that we no longer impose any addition assumptions on their relative sizes (such as $X\leq Z\leq Y$), and thus we deal here with a \emph{general Dynkin game} $\GDG (X,Y,Z)$. As in Remark \ref{remec00}, without loss of generality, we may and do assume that $X_T=Y_T=Z_T$.

 However, since the processes $X, Y$ and $Z$ are now unrestricted, it is easy to construct a Dynkin game without a Nash equilibrium. Our aim in this subsection is to identify necessary and sufficient conditions for the following property:
\begin{gather}
\textit{For all $t=0,1,\ldots,T$,
the Dynkin game $\GDG_t(X,Y,Z)$ admits a Nash equilibrium.}\label{eqem005}
\end{gather}
The idea is to emulate the progression of the previous subsection, while replacing the inequalities $X\leq Z\leq Y$ by a general set of sufficient conditions. When analysing the existence of a Nash equilibrium, we will employ the backward induction argument, as we did in the proof of Proposition \ref{pro:22}. The key argument thus boils down to the thorough analysis of the embedded single period game,
which starts at time $t$ and is either stopped immediately or it is terminated on the next date.

To motivate the construction of the value process candidate in Definition \ref{def:vp2}, let us temporarily assume there exists a value process $V^*$ for the Dynkin game with the payoff process $R$ given by \eqref{eqee00}. Also, let $\tau^*_t, \si^*_t$ be any pair of optimal stopping times for the game starting at time $t$, so that
\begin{gather}\label{eqem011}
V^*_t = \E_\P(R(\tau_t^*,\si_t^*)\,|\,\Filt_{t}).
\end{gather}
Let us denote $P_t:=\E_\P \big( V^*_{t+1} \,|\,  \Filt_{t}\big)$. The next lemma deals with the single period embedded game.

\begin{lemma} \label{lemee05v}
The Nash equilibrium property of a pair $(\tau_t^*,\si_t^*)$ of stopping times is equivalent to the following conditions:
\begin{alignat*}{2}
Y_t\leq V^*_t&=Z_t\leq X_t& \quad\mbox{on}\quad &\{\tau^*_t=t,\si^*_t=t\},\\
P_t\leq V^*_t&=X_t\leq Z_t& \quad\mbox{on}\quad &\{\tau^*_t=t,\si^*_t>t\},\\
Z_t\leq V^*_t&=Y_t\leq P_t& \quad\mbox{on}\quad &\{\tau^*_t>t,\si^*_t=t\},\\
X_t\leq V^*_t&=P_t\leq Y_t& \quad\mbox{on}\quad &\{\tau^*_t>t,\si^*_t>t\}.
\end{alignat*}
\end{lemma}

\begin{proof}
We note that, when written out in full according to definition \eqref{eqee00} of $R$, there are four cases
to examine:
\begin{alignat*}{2}
V^*_t&=Z_t& \quad\mbox{on}\quad &\{\tau^*_t=t,\si^*_t=t\},\\
V^*_t&=X_t& \quad\mbox{on}\quad &\{\tau^*_t=t,\si^*_t>t\},\\
V^*_t&=Y_t& \quad\mbox{on}\quad &\{\tau^*_t>t,\si^*_t=t\},\\
V^*_t&=P_t& \quad\mbox{on}\quad &\{\tau^*_t>t,\si^*_t>t\}.
\end{alignat*}
The stated conditions now follow easily from the definition of the Nash equilibrium.
\end{proof}

Let us write $L_t:=Z_t \wedge X_t$ and $U_t:=Y_t\vee Z_t$, so that $L_t \leq Z_t \leq U_t$
for $t=0,1,\ldots,T$.

\begin{lemma} \label{lemee05n}
Assume that $V^*$ is the value process for $\GDG (X,Y,Z)$ and $\tau^*_t, \si^*_t$ are optimal stopping times for $\GDG_t(X,Y,Z)$. Then: \hfill \break (i) $L_t\leq V^*_t \leq U_t$;
\hfill \break (ii)  $V^*_t=L_t$ on $\{\tau^*_t=t\}$ and $V^*_t=U_t$ on $\{\si^*_t=t\}$;
\hfill \break (iii)  $L_t\leq\E_\P \big( V^*_{t+1} \,|\,  \Filt_{t}\big)\leq U_t$  on the event $\{\tau^*_t \wedge \si^*_t>t\}$.
\end{lemma}

\begin{proof}
 From Lemma \ref{lemee05v}, we deduce easily that $V^*_t$ is always bounded below by $L_t:=Z_t \wedge X_t$ and from above by $U_t:=Y_t\vee Z_t$, so that part (i) is valid.
 This makes sense intuitively since $X_t$ and $Z_t$  ($-Y_t$ and $-Z_t$, resp.) are the possible payoffs of the max-player
 (the min-player, resp.) if he stops at time $t$. Parts (ii) and (iii) also follow easily from Lemma \ref{lemee05v}.
\end{proof}

We note that these behaviours of $L, U, V^*, \tau^*_t, \si^*_t$ are reminiscent of Definition \ref{def:vp1} if processes $X$ and $Y$ are replaced by $L$ and $U$, respectively. This observation furnishes a strong motivation for the following definition.

\begin{definition} \label{def:vp2}
The process $V$ is defined by setting $V_T = Z_T$ and, for any $t=0,1,\dots ,T-1$,
\begin{gather} \label{eq:cb3}
V_{t} := \min \Big \{ U_t ,\, \max \big\{ L_t, \E_\P  ( V_{t+1}\,|\,\Filt_{t}) \big\} \Big\}= \max \Big \{ L_t ,\, \min \big\{ U_t, \E_\P  ( V_{t+1}\,|\,\Filt_{t}) \big\} \Big\},
\end{gather}
where $L:=X \wedge Z$ and $U:= Y\vee Z$. For any fixed $t=0,1,\ldots,T$, the stopping times $\tau^*$
and $\sigma^* $ from $\in\Tstrat_{[t,T]}$ are given by
\begin{gather}
\label{eqee131} \tau^{*}_{t} := \min \big\{ u \in \{ t, t+1, \dots , T \} \,|\, V_u = L_u \big\},\\
\label{eqee132} \si^{*}_{t} :=  \min \big\{ u \in \{ t, t+1, \dots , T \} \,|\, V_u = U_u \big\}.
\end{gather}
\end{definition}

In the remainder of this section, the process $V$ and stopping times $\tau^*_t, \si^*_t$ are as specified by Definition \ref{def:vp2}.
To justify Definition \ref{def:vp2}, we will show in Lemma \ref{lemeo12} that the process $V$ given by \eqref{eq:cb3} is in fact the unique candidate for the value process of the general zero-sum Dynkin game $\GDG (X,Y,Z)$. Of course,
the existence of the value process for $\GDG (X,Y,Z)$ is not yet ensured and in fact some additional
conditions are needed to achieve this goal (see Assumption \ref{assen40}).

Since $L\leq Z\leq U$, it is clear that the second equality in \eqref{eq:cb3} holds and, for $t=0,1,\ldots,T$,
\begin{gather} \label{eqen21}
L_t\leq V_t\leq U_t.
\end{gather}
Let the modified payoff $\wt R$ be given by the following expression
\begin{gather}\label{eqen22}
\wt R(\tau , \si ) := \I_{\{ \tau < \si \}}\, L_{\tau }  +  \I_{\{ \si  < \tau \}}\,
U_{\si }  +  \I_{\{ \si  = \tau \}}\, Z_{\si }.
\end{gather}
Then the analysis of the previous section shows that
\begin{gather}\label{eqen23}
V_t=\E_\P \big(\wt R(\tau^{*}_{t},\si^{*}_{t} )\,|\,  \Filt_{t}\big),
\end{gather}
and Proposition \ref{pro:22} implies that $(\tau^{*}_{t},\si^{*}_{t} )$ is a Nash equilibrium of the standard zero-sum Dynkin game $\SDG_t(L,U,Z)$ associated with the payoff process $\wt R$.
Obviously, this does not mean that they also provide solution to the general zero-sum Dynkin
game $\GDG (X,Y,Z)$ with the payoff process $R$. Nevertheless, the following lemma shows that $V$ is the appropriate candidate of the value process for $\GDG (X,Y,Z)$.

\begin{lemma}\label{lemeo12} For $t=0,1,\ldots,T$, the following properties are valid:
\hfill \break (i)  For any fixed $\tau_t,\si_t\in\Tstrat_{[t,T]}$, there exist $\htau_t, \hsi_t \in\Tstrat_{[t,T]}$ such that
\begin{gather}\label{eqeo121}
R(\htau_t , \si_t ) \geq \wt R(\tau_t , \si_t ) \geq R(\tau_t , \hsi_t ).
\end{gather}
 (ii)  The variable $V_t$ lies between the minimax and the maximin values of $\GDG_t(X,Y,Z)$ so that
\begin{gather}\label{eqeo122}
\essinf_{\si_t \in\STOP_{[t,T]}} \, \esssup_{\tau_t \in\STOP_{[t,T]}}
\E_\P \big(R(\tau_t , \si_t )\,|\,  \Filt_{t}\big) \geq V_t \geq
\esssup_{\tau_t \in\STOP_{[t,T]}} \, \essinf_{\si_t \in\STOP_{[t,T]}}
\E_\P \big(R(\tau_t , \si_t )\,|\,  \Filt_{t}\big).
\end{gather}
(iii)  If the $\GDG_t(X,Y,Z)$ has a value then it equals to $V_t$.
\end{lemma}

\begin{proof} (i) We will only prove the upper inequality of \eqref{eqeo121}, as the lower inequalities follows by symmetry. To choose a stopping time $\htau$ such that $R(\htau_t , \si_t ) \geq \wt R(\tau_t , \si_t ) $, we first compare $R(\tau_t , \si_t )$ and $\wt R(\tau_t , \si_t ) $. On the following events, $R(\tau_t , \si_t ) \geq\wt R(\tau_t , \si_t )$ is automatically satisfied.
\begin{align*}
{\{\tau_t = \si_t\}},&\quad R(\tau_t , \si_t ) = Z_{\tau_t}=\wt R(\tau_t , \si_t );\\
{\{\tau_t < \si_t\}},&\quad R(\tau_t , \si_t ) = X_{\tau_t}\geq L_{\tau_t} =\wt R(\tau_t , \si_t );\\
\{ \si_t < \tau_t ,\, Y_{\si_t}\geq Z_{\si_t}\},&\quad R(\tau_t , \si_t )=Y_{\si_t}=U_{\si_t} = \wt R(\tau_t , \si_t ).
\end{align*}
The problem arises on the event $\{\si_t < \tau_t ,\, Z_{\si_t}>Y_{\si_t}\}$, since then
\[
R(\tau_t , \si_t )=Y_{\si_t} < U_{\si_t} = \wt R(\tau_t , \si_t ).
\]
Let us modify $\tau$ by setting
\begin{gather*}
\htau = \si_t \I_{\{\si_t < \tau_t, Z_{\si_t}>Y_{\si_t}\}} + \tau_t \big(1- \I_{\{\si_t < \tau_t, Z_{\si_t}>Y_{\si_t}\}}\big).
\end{gather*}
Then $\htau$ is indeed an $\FF$-stopping time, since the event $\{\si_t < \tau_t, Z_{\si_t}>Y_{\si_t}\}$ belongs to $\Filt_{\si_t\wedge \tau_t}$. Furthermore, on the event $\{\si_t < \tau_t, Z_{\si_t}>Y_{\si_t}\}$ we have that
\[
R(\htau_t , \si_t ) = R(\si_t , \si_t ) = Z_{\si_t} = U_{\si_t} =\wt R(\tau_t , \si_t )
\]
and thus for the stopping time $\htau $ the left-hand side inequality in \eqref{eqeo121} is satisfied.

\noindent (ii) Again, we only show the upper inequality of \eqref{eqeo122}. By Proposition \ref{pro:22}, $V_t$ is the value of $\SDG_t(L,U,Z)$ associated with $\wt R$. Let $(\tau^*_t, \si^*_t)$ be a Nash equilibrium of $\SDG_t(L,U,Z)$. Hence we have, for any $\si_t\in\Tstrat_{[t,T]}$,
\begin{gather*}
\E_\P \big(\wt R(\tau^*_t , \si_t )\,|\,  \Filt_{t}\big) \geq \E_\P \big(\wt R(\tau^*_t , \si^*_t )\,|\,  \Filt_{t}\big) = V_t.
\end{gather*}
By part (i), there exists $\htau_t\in\Tstrat_{[t,T]}$ such that $R(\htau_t , \si_t )\geq \wt R(\tau^*_t , \si_t )$. Consequently,
\begin{gather}\label{lemeo1221}
\esssup_{\tau_t \in\STOP_{[t,T]}}
\E_\P \big(R(\tau_t , \si_t )\,|\,  \Filt_{t}\big) \geq
\E_\P \big(R(\htau_t , \si_t )\,|\,  \Filt_{t}\big) \geq
\E_\P \big(\wt R(\tau^*_t , \si_t )\,|\,  \Filt_{t}\big) \geq V_t.
\end{gather}
Since \eqref{lemeo1221} holds for all $\si_t\in\Tstrat_{[t,T]}$, we must have
\begin{gather*}
\essinf_{\si_t \in\STOP_{[t,T]}} \, \esssup_{\tau_t \in\STOP_{[t,T]}}
\E_\P \big(R(\tau_t , \si_t )\,|\,  \Filt_{t}\big) \geq V_t,
\end{gather*}
as required.

\noindent (iii) By the definition of the value (see Definition \ref{defaa03}), if there exists a value $V^*_t$ for $\GDG_t(X,Y,Z)$, it must satisfy
\begin{gather}\label{lemeo1231}
V^*_t=\essinf_{\si_t \in\STOP_{[t,T]}} \, \esssup_{\tau_t \in\STOP_{[t,T]}}
\E_\P \big(R(\tau_t , \si_t )\,|\,  \Filt_{t}\big) =
\esssup_{\tau_t \in\STOP_{[t,T]}} \, \essinf_{\si_t \in\STOP_{[t,T]}}
\E_\P \big(R(\tau_t , \si_t )\,|\,  \Filt_{t}\big).
\end{gather}
In view of part (ii), we conclude that necessarily $V^*_t=V_t$.
\end{proof}

Even though $V$ is the unique value process candidate for $\GDG (X,Y,Z)$, the existence of the value process has not been established. There are two major obstacles to overcome when attempting to apply the backward induction argument similar to Proposition \ref{pro:22} on the payoff process $R$.

First, it is not necessarily true that $V_t=\E_\P \big(R(\tau^{*}_{t},\si^{*}_{t} )\,|\,  \Filt_{t}\big)$. In particular, this equality fails to hold if either of the following occurs:
\begin{alignat}{2}
Z_{\tau^*_t} &= V_{\tau^*_t} < X_{\tau^*_t}\wedge Y_{\tau^*_t}&\quad\text{on the event}\quad&\{\tau^*_t < \si^*_t\},\label{eqen31}\\
Z_{\si^*_t} &= V_{\si^*_t} > X_{\si^*_t}\vee Y_{\si^*_t}&\quad\text{on the event}\quad&\{\si^*_t < \tau^*_t\}.\label{eqen32}
\end{alignat}
 Second, it is possible that $V$ fails to satisfy any of the
necessary conditions on $V^*$ established in Lemma \ref{lemee05v}. An exhaustive check shows that the exceptions are:
\begin{align}
Z_t \leq V_t < X_t\wedge Y_t,\label{eqen33}\\
Z_t \geq V_t > X_t\vee Y_t.\label{eqen34}
\end{align}
It is crucial to observe that the undesirable scenarios may only occur when $V$ is either greater than $X\vee Y$ or less than $X\wedge Y$. Therefore, it is natural to introduce the following additional assumption.

\begin{assumption}\label{assen40}
Let $X, Y$ and $Z$ be $\FF$-adapted integrable processes and let the associated process $V$ be given as in Definition \ref{def:vp2}. We postulate that the processes $X, Y$ and $V$ satisfy, for $t=0,1,\ldots,T$,
\begin{gather}\label{assen41}
X_t\wedge Y_t \leq V_t \leq X_t\vee Y_t .
\end{gather}
\end{assumption}

Assumption \ref{assen40} certainly eliminates the scenarios described in \eqref{eqen31}--\eqref{eqen34}. Since $V$ is defined in terms of $X, Y$ and $Z$, this is really an assumption on $X, Y$ and $Z$, albeit its form is somewhat convoluted, since it also
refers to formula \eqref{eq:cb3}. In the foregoing example, we provide some more explicit conditions that entail
Assumption \ref{assen40}.

\begin{example} (i) Let us first consider the conditions from the previous section: $X_T = Z_T = Y_T$ and $X \leq Z \leq Y$.
In view of \eqref{eqen21}, it is clear that Assumption \ref{assen40} is satisfied since $L=X=X\wedge Y$ and $U=Y=X\vee Y$. This shows that Assumption \ref{assen40} covers the case of the standard zero-sum Dynkin game.

\noindent (ii) Suppose $X, Y$ and $Z$ satisfy $X_T = Z_T = Y_T$ and, for all $t \in [0,T]$,
\begin{gather*}
X_t\wedge Y_t > Z_t \quad \implies \quad X_t\wedge Y_t \leq \E_\P\big(X_{t+1}\wedge U_{t+1} \,|\,\Filt_t\big),\\
X_t\vee Y_t < Z_t \quad \implies \quad X_t\vee Y_t \geq \E_\P\big( L_{t+1} \vee Y_{t+1} \,|\,\Filt_t\big).
\end{gather*}
 One can check that Assumption \ref{assen40} is satisfied.

 \noindent (iii) It should be acknowledged that various generalisations of the standard Dynkin game were studied
 in the literature. In particular,  Ohtsubo \cite{ohtsubo1987nonzero} examined the zero-sum Dynkin game with an infinite time horizon under
 the assumption that
 \begin{gather}\label{accen41}
X_t\wedge Y_t \leq Z_t \leq X_t\vee Y_t .
\end{gather}
Once again, we see that if \eqref{accen41} holds then  Assumption \ref{assen40} is satisfied.
  He established the existence of a Nash equilibrium for the game starting at any date $t$ under the assumption that $(X_t)$ and $(Y_t)$ are mutually independent sequences of i.i.d.\! random variables (see Corollary 3.2 in \cite{ohtsubo1987nonzero}).

As a special case, Ohtsubo  \cite{ohtsubo1987nonzero} considered also the game with the payoff
\begin{gather*} 
R(\tau , \si ) := \I_{\{ \tau < \si \}}\, X_{\tau }  +  \I_{\{ \si  \leq \tau \}}\,
Y_{\si }
\end{gather*}
for arbitrary $\FF$-adapted, integrable processes $X$ and $Y$. Since here $Z=Y$, so that $L = X \wedge Y$ and $U = Y$, it follows
easily from  \eqref{eqen21} that Assumption \ref{assen40} is satisfied.
It can be deduced from Proposition 3.1 in \cite{ohtsubo1987nonzero} that in the finite horizon case
the game admits a Nash equilibrium and
the value process $V^*$ satisfies: $V^*_T = Y_T = Z_T $ and, for $t=0,1,\dots, T-1$,
\begin{gather} \label{xeq:cb3}
V^*_{t} =\, Y_{t} \, \I_{\{ Y_t \leq X_t \}}
+ \min \Big \{ Y_t ,\, \max \big\{ X_t, \E_\P  ( V^*_{t+1}\,|\,\Filt_{t}) \big\} \Big\} \I_{\{ Y_t > X_t \}}.
\end{gather}
This result can be seen as a special case of Proposition \ref{propee10}, since for $Y=Z$ equation \eqref{eq:cb3}
becomes: $V_T =  Z_T = Z_T $ and
\begin{gather*}
V_{t} := \min \Big \{ Y_t ,\, \max \big\{ X_t \wedge Y_t , \E_\P  ( V_{t+1}\,|\,\Filt_{t}) \big\} \Big\},
\end{gather*}
which indeed coincides with \eqref{xeq:cb3}, so that $V = V^*$ where $V^*$ is given by \eqref{xeq:cb3}.
\end{example}

\subsubsection{Sufficiency of Assumption \ref{assen40}}

Our goal is to demonstrate that Assumption \ref{assen40} is the necessary and sufficient condition for \eqref{eqem005} to hold.
 We start by examining the sufficiency of Assumption \ref{assen40}.
To this end, we first prove an auxiliary lemma.

\begin{lemma} \label{lemee05} Under Assumption \ref{assen40}, for each $t=0,1,\ldots,T$, the process $V$ satisfies:
\hfill \break (i)  $\{V_t>Y_t\}\subseteq \{\tau^*_t=t\}$ and $\{V_t<X_t\}\subseteq \{\si^*_t=t\}$;
\hfill \break (ii)  $V_t=\E_\P\big(R(\tau^*_t,\si^*_t)\,|\, \Filt_t\big)$.
\end{lemma}

\begin{proof} (i) For the first inclusion, let us suppose that $V_t>Y_t$. Since Assumption \ref{assen40} states that $V_t$ has to lie in between $X_t$ and $Y_t$, we obtain
\begin{gather}\label{lemee0551}
V_t\leq X_t.
\end{gather}
From \eqref{eqen21} we obtain $V_t\leq U_t=Y_t\vee Z_t$, and thus we must also have
\begin{gather}\label{lemee0552}
V_t\leq Z_t.
\end{gather}
By combining \eqref{lemee0551} with \eqref{lemee0552}, we obtain $V_t\leq X_t\wedge Z_t = L_t$. Moreover, by noting
 that $V_t\geq L_t$ from \eqref{eqen21}, we conclude that $V_t = L_t$ and thus, by \eqref{eqee131}, the equality $\tau^*_t=t$ holds, as required. The second inclusion can be shown using similar arguments.

\noindent (ii) It is sufficient to show
\begin{gather}\label{lemee070}
V_{\tau^*_t\wedge \si^*_t} = R(\tau^*_t,\si^*_t) = \I_{\{ \tau^*_t < \si^*_t \}}\, X_{\tau^*_t }  +  \I_{\{ \si^*_t  < \tau^*_t \}}\,
Y_{\si^*_t }  +  \I_{\{ \si^*_t  = \tau^*_t \}}\, Z_{\si^*_t }.
\end{gather}
On the event $\{ \si^*_t  = \tau^*_t \}$, we have $V_{\si^*_t} = U_{\si^*_t} = L_{\si^*_t} = Z_{\si^*_t}$ as required.
On the event $\{ \si^*_t  < \tau^*_t \}$, we have $V_{\si^*_t} = U_{\si^*_t} = Z_{\si^*_t} \vee Y_{\si^*_t} \geq Y_{\si^*_t}$. If $V_{\si^*_t} > Y_{\si^*_t}$, then from (i), we obtain $\tau^*_t=\si^*_t$, which is a contradiction.
 We thus conclude that $V_{\si^*_t} = Y_{\si^*_t}$, as required.

The case of $\{ \si^*_t  > \tau^*_t \}$ is similar to the case of $\{ \si^*_t  < \tau^*_t \}$. This establishes \eqref{lemee070}.
\end{proof}

 We are now in a position to show that Assumption \ref{assen40} implies the existence of Nash equilibria
 for the family of Dynkin games $\GDG_t(X,Y,Z),\, t=0,1, \dots , T$.

\begin{proposition}\label{propee10}
Let the process $V$ and the stopping times $\tau^{*}_{t},\si^{*}_{t} $  be given as in
Definition \ref{def:vp2}. If Assumption \ref{assen40} holds then for arbitrary stopping times $\tau_t, \si_t \in T_{[t,T]}$,
\begin{gather*} 
\E_\P \big( R(\tau^{*}_{t},\si_t )\,|\,  \Filt_{t}\big) \geq \E_\P \big(
R(\tau^{*}_{t},\si^{*}_t )\,|\,  \Filt_{t}\big) \geq \E_\P \big(
R(\tau_t ,\si^{*}_{t})\,|\,  \Filt_{t}\big)
\end{gather*}
and thus $(\tau^{*}_{t},\si^{*}_{t})$ is a Nash equilibrium of $\GDG_t(X,Y,Z)$.
Furthermore, the process $V$ is the value process of $\GDG (X,Y,Z)$, that is, for every $t=0,1,\dots ,T$,
\begin{align*} 
V_t &= \essinf_{\si_t \in\STOP_{[t,T]}} \,
\esssup_{\tau_t \in\STOP_{[t,T]}} \E_\P \big(R(\tau_t ,\si_t )\,|\,
\Filt_{t}\big)=\EP \big(R(\tau^{*}_t ,\si^{*}_t)\,|\, \Filt_{t}\big)
 \\ &=  \esssup_{\tau_t \in\STOP_{[t,T]}} \, \essinf_{\si_t \in\STOP_{[t,T]}}
\E_\P \big(R(\tau_t ,\si_t )\,|\, \Filt_{t}\big) = V_t^* , \nonumber
\end{align*}
and $\tau^{*}_t, \si^{*}_t$ are the optimal stopping times as of time $t$. In particular, $V^*_T = Z_T$
and for any $t=0,1,\dots ,T-1$,
\begin{gather*} 
V_{t}^* = \min \Big \{U_t ,\, \max \big\{ L_t, \E_\P  ( V_{t+1}^*\,|\,\Filt_{t}) \big\} \Big\}.
\end{gather*}
\end{proposition}

\begin{proof}
The arguments used in this proof will be very similar to the ones from Proposition \ref{pro:22}, with the help of Lemma \ref{lemee05}.
In view of part (ii) in Lemma \ref{lemee05}, it is sufficient to show that
\begin{gather}\label{eqee14}
\E_\P \big( R(\tau^{*}_{t},\si_{t} )\,|\,  \Filt_{t}\big) \geq V_{t}
\geq \E_\P \big( R(\tau_{t} ,\si^{*}_{t})\,|\,  \Filt_{t}\big).
\end{gather}
To this end, we proceed by backward induction. The inequalities \eqref{eqee14} clearly hold for $t=T$. Assume now that they are true for some $t \leq T$. We wish to prove that, for arbitrary  $\tau_{t-1},\si_{t-1}\in T_{[t-1,T]}$,
\begin{gather}\label{eqee15}
\E_\P \big( R(\tau^{*}_{t-1},\si_{t-1} )\,|\,  \Filt_{t-1}\big) \geq V_{t-1} \geq \E_\P \big(
R(\tau_{t-1} ,\si^{*}_{t-1})\,|\,  \Filt_{t-1}\big).
\end{gather}

We will establish the upper bound of \eqref{eqee15}, the lower bound follows by the symmetry of the Dynkin game. Again the argument can be split into two main cases: either $\GDG_{t-1}(X,Y,Z)$ is stopped at time $t-1$ or it is continued to time $t$ and the induction hypothesis becomes relevant. As before, for any $\tau_{t-1},\si_{t-1}\in T_{[t-1,T]}$, we denote $\wt \tau_{t-1}:=\tau_{t-1}\vee t$ and $\wt \si_{t-1}:=\si_{t-1}\vee t$, so that $\wt \tau_{t-1}, \wt \si_{t-1} \in \STOP_{[t,T]}$.

Let us examine the case where the game is stopped at time $t-1$. On the event $\{\tau^*_{t-1}=t-1\}$, we have
\begin{gather}\label{eqee16}
 \E_\P \big(R(\tau^{*}_{t-1},\si_{t-1} )\,|\,  \Filt_{t-1}\big)  =
  \I_{\{\si_{t-1} > t \}}X_{t-1} + \I_{\{ \si_{t-1} = t \}} Z_{t-1}
\geq L_{t-1} = V_{t-1}.
\end{gather}
On the event $\{\si_{t-1}=t-1<\tau^*_{t-1}\}$, we obtain $\E_\P \big(R(\tau^{*}_{t-1},\si_{t-1} )\,|\,  \Filt_{t-1}\big)  = Y_{t-1}$. If $V_{t-1}> Y_{t-1}$ then, by part (i) in Lemma \ref{lemee05}, we have that $\tau^*_{t-1}= t-1$, which is a contradiction. Hence $V_{t-1}\leq Y_{t-1}$ and thus
\begin{align}\label{eqee18}
\E_\P \big(R(\tau^{*}_{t-1},\si_{t-1} )\,|\,  \Filt_{t-1}\big) = Y_{t-1} \geq V_{t-1}.
\end{align}
Let us now assume that the game is not stopped at time $t-1$, that is, we now consider the event $\{\tau^*_{t-1}\geq t,\si_{t-1}\geq t\}$. We observe that here $\tau^*_{t-1}=\tilde \tau^*_{t-1}=\tau^*_{t}$ and $V_{t-1}>L_{t-1}$, so that  \eqref{eq:cb3} yields
\begin{gather}\label{eqee17}
V_{t-1}=\min\big\{U_{t-1},\E_\P \big( V_{t} \,|\,  \Filt_{t-1}\big)\big\}.
\end{gather}
Consequently,
\begin{align}
\E_\P \big(R(\tau^{*}_{t-1},\si_{t-1} )\,|\,  \Filt_{t-1}\big)  &= \E_\P \big( R(\tilde \tau^{*}_{t-1},\tilde \si_{t-1} )\,|\,  \Filt_{t-1}\big) \nonumber\\
&= \E_\P \big( \E_\P(R(\tau^{*}_{t},\tilde \si_{t-1} )\,|\,  \Filt_t\big) \,|\,  \Filt_{t-1}\big)\nonumber\\
&\geq \E_\P \big( V_{t} \,|\,  \Filt_{t-1}\big) \label{eqee19}\\
&\geq \min\big\{U_{t-1},\E_\P \big( V_{t} \,|\,  \Filt_{t-1}\big)\big\} \nonumber\\
&=V_{t-1}.\label{eqee20}
\end{align}
Note that \eqref{eqee19} follows from the induction hypothesis \eqref{eqee14} while \eqref{eqee20} follows from \eqref{eqee17}.

Combining \eqref{eqee16}, \eqref{eqee18} and \eqref{eqee20} gives the upper inequality of \eqref{eqee15}. As already mentioned, the lower inequality of \eqref{eqee15} follows by symmetry. Therefore, $(\tau^*_t,\si^*_t)$ is a Nash  equilibrium of $\GDG_t(X,Y,Z)$ and $V_t=\E_\P \big(R(\tau^{*}_{t},\si^{*}_{t} )\,|\,  \Filt_{t}\big)$ is the value.
\end{proof}

\subsubsection{Necessity of Assumption \ref{assen40}}

To prove that Assumption \ref{assen40} is also a necessary condition for property \eqref{eqem005} to hold, it suffices to show that if this assumption is violated then there exists $t$ such that the general Dynkin game $\GDG_t(X,Y,Z)$ does not have a Nash equilibrium.
Recall that the process $V$ in Definition \ref{def:vp2} was originally chosen to be the value process of the Dynkin game $\SDG (L,U,Z)$ associated with the payoff process
\begin{gather*} 
\wt R(\tau , \si ) := \I_{\{ \tau < \si \}}\, L_{\tau }  +  \I_{\{ \si  < \tau \}}\,
U_{\si }  +  \I_{\{ \si  = \tau \}}\, Z_{\si }.
\end{gather*}
Next, in Lemma \ref{lemeo12}, it was shown that if the Dynkin game $\GDG (X,Y,Z)$ associated with the payoff process
\begin{gather*} 
R(\tau , \si ) := \I_{\{ \tau < \si \}}\, X_{\tau }  +  \I_{\{ \si  < \tau \}}\,
Y_{\si }  +  \I_{\{ \si  = \tau \}}\, Z_{\si }.
\end{gather*}
has a value process then it has to be a version of $V$.
Finally, we formulated Assumption \ref{assen40}, which was shown to ensure that $\GDG (X,Y,Z)$ has a value process.

\begin{proposition}\label{propeo20}
Suppose that Assumption \ref{assen40} is violated at time $t\in[0,T]$, that is,
\begin{gather}\label{eqeo20}
\P \big( \{V_t < X_t \wedge Y_t\} \cup \{V_t > X_t \vee Y_t\} \big) > 0 .
\end{gather}
Then $\GDG_t(X,Y,Z)$ does not have a Nash equilibrium.
\end{proposition}

\begin{proof}
Since $X_T=Y_T=Z_T=V_T$ then manifestly \eqref{eqeo20} cannot occur when $t=T$. Assume, for the sake of contradiction, that \eqref{eqeo20} holds for some $t<T$ and there is a Nash equilibrium $(\tau^*_t,\si^*_t)$. Then, by part (iii) in Lemma \ref{lemeo12}, $V_t=V^*_t=\E_\P \big(R(\tau^*_t , \si^*_t )\,|\,  \Filt_{t}\big)$ is the value of $\GDG_t(X,Y,Z)$.

 Assume now that either $\P (\{V_t < X_t \wedge Y_t\})>0$ or $\P (\{V_t > X_t \vee Y_t\})>0$. First, consider the event $\{V_t < X_t \wedge Y_t\}$. Neither $\tau^*_t=t<\si^*_t$ nor $\tau^*_t=t<\si^*_t$ can occur, as otherwise we would have that either $V_t=X_t$ or $V_t=Y_t$, respectively. If $\tau^*_t=\si^*_t=t$ then $R(t , \si^*_t )=Y_t > V_t$, contradicting the property of the Nash equilibrium. If $t<\tau^*_t\wedge \si^*_t$ then $R(t , \si^*_t )=X_t > V_t$, which is also a contradiction.

The same argument can be made for the event $\{V_t > X_t \vee Y_t\}$. We conclude there cannot be a Nash equilibrium for the
Dynkin game starting at time $t$ if condition \eqref{eqeo20} is valid.
\end{proof}

Propositions \ref{pro:22} and \ref{propeo20} can be combined into the following main result of this section, which explicitly states the condition needed for the existence of a Nash equilibrium for arbitrary payoff processes $X, Y$ and $Z$. Theorem \ref{coreo25} is thus an essential generalisation of  Proposition \ref{pro:22} for the standard zero-sum Dynkin game, which only addressed the case of $X\leq Z\leq Y$.

\begin{theorem}\label{coreo25}
Let $X, Y$ and $Z$ be $\FF$-adapted, integrable processes and let the process $V$ be given by: $V_T:=Z_T$ and, for $t=0,1,\ldots,T-1$,
\begin{gather*}
V_{t} := \min \Big \{ U_t ,\, \max \big\{ L_t, \E_\P  ( V_{t+1}\,|\,\Filt_{t}) \big\} \Big\}
\end{gather*}
 where $L=X\wedge Z$ and $U=Y\vee Z$. The inequality
\begin{gather*}
X_t\wedge Y_t \leq V_t \leq X_t\vee Y_t
\end{gather*}
holds for all $t=0,1,\ldots,T$ if and only if the Dynkin game $\GDG_t(X,Y,Z)$ starting at time $t$ and associated with the payoff
\begin{gather*} 
R(\tau , \si ) := \I_{\{ \tau < \si \}}\, X_{\tau }  +  \I_{\{ \si  < \tau \}}\,
Y_{\si }  +  \I_{\{ \si  = \tau \}}\, Z_{\si }.
\end{gather*}
has a Nash equilibrium for all $t=0,1,\ldots,T$.
\end{theorem}

\begin{remark}
Theorem \ref{coreo25} answers the question regarding the existence of Nash equilibrium in for the set of Dynkin games starting at all times $t=0,1,\ldots,T$. For a Dynkin game starting at a particular value of $t$, the exact result is unclear. Assumption \ref{assen40} certainly provides a sufficient condition, but it is not a necessary condition.
\end{remark}

\section{Continuous-Time Dynkin Games}\label{sec2f}

In this section, we deal with continuous-time versions of two-person, zero-sum stopping games
with a finite time horizon. As previously, we focus on conditions under which the game admits
a Nash equilibrium.

\subsection{Standard Dynkin Game}\label{sec3.21}

In this preliminary subsection, we re-examine the standard zero-sum Dynkin game in continuous-time. We first recall two definitions.

%
	%
%

\begin{definition}[Value]\label{defaa03}
Consider a two-player, zero-sum game $\Game$ with strategy spaces $\Strat^{1}$ and $\Strat^{2}$ and payoff function $V$. It is said to have a \emph{value} $V^*$ if  
\begin{gather*}
V^*=\essinf_{\si \in\Strat^2} \, \esssup_{\tau_t \in\Strat^1}
V(\tau , \si ) =
\esssup_{\tau_t \in\Strat^1} \, \essinf_{\si \in\Strat^2}
V(\tau , \si ).
\end{gather*}
\end{definition}

\begin{definition}\label{defep01}
Suppose that a game $\Game$ has a value $V^{*}$. For $\ep\geq 0$, an $\ep$-\emph{optimal strategy} $\tau^\ep\in\Strat^1$ for the max-player guarantees the payoff to within $\ep$ of the value. In other words
\begin{gather}\label{eqep011}
\essinf_{\si\in\Strat^2} V(\si,\tau^\ep) \geq  V^{*}-\ep.
\end{gather}
Similarly, an $\ep$-\emph{optimal strategy} $\si^\ep\in\Strat^2$ for the min-player satisfies
\begin{gather}\label{eqep0112}
\esssup_{\tau\in\Strat^1} V(\si^\ep,\tau) \leq  V^{*}+\ep.
\end{gather}
A strategy profile $(\si,\tau)$ is is called an \emph{$\ep$-equilibrium} if it consists of $\ep$-optimal strategies for both players.
\end{definition}


Note that a Nash equilibrium is a 0-equilibrium.
The following result is easy to prove and thus the proof is omitted.

\begin{proposition}\label{propep05}
In a two-player, zero-sum game $\Game$, the following statements are equivalent.
\hfill \break (i)  The game has a value for both players.
\hfill \break (ii)  For all $\ep>0$, there exist $\ep$-optimal strategies for both players.
\hfill \break (iii)  For all $\ep>0$, there exists an $\ep$-equilibrium.
\hfill \break (iv)  For all $\ep>0$, there exists a real number $v^\ep$ and a strategy profile $(\si^\ep, \tau^\ep)$ such that
\begin{gather}\label{eqep054}
\essinf_{\si\in\Strat^{2}} V^1(\si,\tau^\ep) \geq v^\ep \geq \esssup_{\tau\in\Strat^{1}} V^1(\si^\ep,\tau).
\end{gather}
\end{proposition}

Let the time parameter $t\in[0,T]$ be continuous and the filtration $\FF$ be right-continuous. Let $X, Y$ and $Z$ be $\FF$-adapted,
c\`adl\`ag processes satisfying the usual integrability condition.
Consider the \emph{standard Dynkin game} $\SDG_t(X,Y,Z)$ starting at $t$ and with payoff given by
\begin{gather} \label{eqff11}
R(\tau , \si ) = \I_{\{ \tau  < \si \}}\, X_{\tau }  +  \I_{\{ \si  < \tau \}}\,
Y_{\si } + \I_{\{ \si  = \tau \}}\, Z_{\si }
\end{gather}
where $\tau , \si $ are $\FF$-stopping times and $X\leq Z\leq Y$. We denote by $\SDG (X,Y,Z)$ the family of Dynkin games $\SDG_t(X,Y,Z),\, t \in [0,T]$. As in Section \ref{sec2a}, without the loss of generality, we set $X_T=Y_T=Z_T$.

The case of the standard zero-sum continuous-time Dynkin game has been studied by several authors, for example, Lepeltier and Maingueneau \cite{lepeltier1984jeu}.
The following result summarises some of these findings.

\begin{theorem}\label{thmev11}
 Consider the standard zero-sum Dynkin game $\SDG (X,Y,Z)$ associated with the payoff $R$ given by formula \eqref{eqff11}.

\noindent (i)  For any $t \in [0,T]$, the standard zero-sum Dynkin game $\SDG_t(X,Y,Z)$ has a value $V^*_t$ satisfying
\begin{align} \label{eqev111}
V^*_t &= \essinf_{\si_t \in\STOP_{[t,T]}} \,
\esssup_{\tau_t \in\STOP_{[t,T]}} \E_\P \big(R(\tau_t ,\si_t )\,|\,
\Filt_{t}\big) \\
&=  \esssup_{\tau_t \in\STOP_{[t,T]}} \, \essinf_{\si_t \in\STOP_{[t,T]}} \E_\P \big(R(\tau_t ,\si_t )\,|\,
\Filt_{t}\big). \nonumber
\end{align}
The value process $V^*$ of $\SDG (X,Y,Z)$  can be chosen to be right-continuous.

\noindent (ii) For any $t \in [0,T]$ and any $\ep>0$, the pair of $\FF$-stopping times $(\tau^\ep_t, \si^\ep_t)\in \STOP_{[t,T]} \times \STOP_{[t,T]}$ defined by
\begin{gather}
\label{eqev112} \si^\ep_t:=\inf\{u \geq t : Y_u \leq V^*_u+\ep\},\quad \tau^\ep_t:=\inf\{u \geq t : X_u \geq V^*_u-\ep\}
\end{gather}
are $\ep$-optimal strategies satisfying
\begin{gather}\label{eqev114}
\essinf_{\si_t \in\STOP_{[t,T]}} \E_\P \big( R(\tau^{\ep}_{t},\si_t )\,|\,  \Filt_{t}\big)+\ep \geq V^*_t \geq \esssup_{\tau_t \in\STOP_{[t,T]}} \E_\P \big(
R(\tau_t ,\si^{\ep}_{t})\,|\,  \Filt_{t}\big)-\ep.
\end{gather}

\noindent (iii) If we further assume that $X$ and $-Y$ are left upper semi-continuous (only have positive jumps), then $\SDG_t(X,Y,Z)$ has a Nash equilibrium $(\tau^*_t, \si^*_t)\in \STOP_{[t,T]} \times \STOP_{[t,T]}$ satisfying
\begin{gather}
\label{eqev116} \si^*_t:=\lim_{\ep\to 0} \si^\ep_t,\quad \tau^*_t:=\lim_{\ep\to 0} \tau^\ep_t
\end{gather}
and
\begin{gather}\label{eqev117}
\E_\P \big( R(\tau^{*}_{t},\si_t )|  \Filt_{t}\big) \geq \E_\P \big(
R(\tau^{*}_{t},\si^{*}_t )|  \Filt_{t}\big) =V^*_t \geq \E_\P \big(
R(\tau_t ,\si^{*}_{t})|  \Filt_{t}\big),\ \forall\, \tau_t, \si_t \in \STOP_{[t,T]}.
\end{gather}
\end{theorem}

\begin{proof}
Theorem \ref{thmev11} summarises well know results and thus its proof is omitted.
\end{proof}

 Observe that there may be other $\ep$-optimal strategy pairs (resp. Nash equilibria) than the ones specified by \eqref{eqev112} (resp. \eqref{eqev116}). Also $\si^*_t, \tau^*_t$ do not necessarily coincide with stopping times $\si^0_t, \tau^0_t$, which are defined by setting $\ep=0$ in \eqref{eqev112}, that is,
\begin{gather*}
\si^0_t:=\inf\{u \geq t : Y_u \leq V^*_u\},\quad \tau^0_t:=\inf\{u \geq t : X_u \geq V^*_u\}.
\end{gather*}
In general, we have that $\si^*_t\leq \si^0_t$ and $\tau^*_t\leq \tau^0_t$.

\subsubsection{Auxiliary Results}

Before moving on to the next subsection, we will first establish several auxiliary properties, which are
consequences of Theorem \ref{thmev11}.

\begin{lemma}\label{lemev21} (i) For any $t\in[0,T]$, we have that
$X_t \leq V^*_t \leq Y_t.$ \hfill \break
(ii) For $\ep\geq0$, let $\tau^\ep_t, \si^\ep_t$ be as defined in \eqref{eqev112}. Then
\begin{gather}\label{eqev211}
X_{\tau^\ep_t} \geq V^*_{\tau^\ep_t}-\ep,\quad Y_{\si^\ep_t} \leq V^*_{\si^\ep_t}+\ep .
\end{gather}
\end{lemma}

\begin{proof} (i) The lower bound follows from
\[
V^*_t = \esssup_{\tau_t \in\STOP_{[t,T]}} \, \essinf_{\si_t \in\STOP_{[t,T]}} \E_\P \big(R(\tau_t ,\si_t )\,|\,
\Filt_{t}\big) \geq \essinf_{\si_t \in\STOP_{[t,T]}} \E_\P \big(R(t ,\si_t )\,|\,
\Filt_{t}\big) \geq X_t.
\]
The upper bound can be shown similarly. Part (ii) follows immediately from the right-continuity of $V^*$, $X$ and $Y$.
\end{proof}

\begin{lemma}\label{lemev51}
Let $G$ and $H$ be integrable progressively measurable processes. Suppose $G$ is right lower semicontinuous and $H$ is right-continuous. If for each $t\in[0,T]$, $G_t\leq H_t$ a.s., then for all $\rho\in\Tstrat_{[0,T]}$, $G_\rho\leq H_\rho$ a.s..
\end{lemma}

\begin{proof}
Choose a sequence of decreasing stopping times $\rho_n$ which takes countably many values and converge to $\rho$. Then
\[
G_\rho \leq \lim_{n\to\infty} G_{\rho_n} \leq \lim_{n\to\infty} H_{\rho_n} = H_{\rho},
\]
as required. \end{proof}

For a fixed $\si\in\Tstrat_{[0,T]}$, the process $R^\si_t:=R(t,\si)$ is right lower semicontinuous, but not necessarily continuous. So let us define the right-continuous process
\[
\hat R^\si_t:=X_t \I_{\{t < \si\}} + Y_\si \I_{\{\si \leq t\}}.
\]
Since $Y\geq Z\geq X$, we have that $\hat R^\si_t \geq R^\si_t$. Consequently, by Lemma \ref{lemev51}, $\hat R^\si_\rho \geq R^\si_\rho$ for all $\rho\in\Tstrat_{[0,T]}$. On the other hand, since $\hat R^\si_\rho = R(\rho \I_{\{\rho < \si^\ep_t\}} + T \I_{\{\rho \geq \si^\ep_t\}},\si)$, the following Snell envelope of $R^\si_t$ and $\hat R^\si_t$
\begin{gather}\label{eqev53}
Q^{\si}_t:= \esssup_{\rho\in\Tstrat_{[t,T]}} \E_\P \big(R^\si_\rho\,|\,  \Filt_{t}\big) = \esssup_{\rho\in\Tstrat_{[t,T]}} \E_\P \big(\hat R^\si_\rho\,|\,  \Filt_{t}\big)
\end{gather}
is a well-defined right-continuous supermartingale. It follows immediately from Lemma \ref{lemev51} that, for all $\tau\in\Tstrat_{[0,T]}$,
\begin{gather}\label{eqev54}
R(\tau,\si) \leq Q^{\si}_\tau.
\end{gather}
The process $Q^{\si}$ can also be used to demonstrate properties of $V^*$.

\begin{proposition}\label{propev20}
Let $V^*$ be as defined in \eqref{eqev111} and $Q^{\si}$ be as defined in \eqref{eqev53}.

\noindent  (i)  For all $\si, \tau\in\Tstrat_{[0,T]}$,
\begin{gather}\label{eqev55}
V^*_{\si\wedge\tau} \leq Q^\si_\tau.
\end{gather}

\noindent (ii) For $\ep\geq 0$, let $\hsi_t \in \STOP_{[t,T]}$ be an arbitrary $\ep$-optimal strategy for the min-player in $\SDG_t(X,Y,Z)$ and $\tau_t\in \STOP_{[t,T]}$ be any $\FF$-stopping time. If $P$ is an $\Filt_T$-measurable random variable satisfying $P\leq Q^{\hsi_t}_{\tau_t}$, then
\begin{gather}\label{eqev200}
\E_\P \big(P \,|\,  \Filt_{t}\big) \leq V^*_t + \ep.
\end{gather}

\noindent (iii) For $\ep\geq 0$, let $(\htau_t, \hsi_t)\in \STOP_{[t,T]} \times \STOP_{[t,T]}$ be an arbitrary pair of $\ep$-optimal strategies of $\SDG_t(X,Y,Z)$. Then for all $\si_t,\tau_t\in \Tstrat_{[t,T]}$,
\begin{gather}\label{eqev201}
\E_\P \big( V^*_{\htau_t\wedge\si_t} \,|\,  \Filt_{t}\big)+\ep \geq V^*_t \geq \E_\P \big(
V^*_{\hsi_t\wedge\tau_t} \,|\,  \Filt_{t}\big)-\ep.
\end{gather}

\noindent (iv) If $(\htau_t, \hsi_t)\in \STOP_{[t,T]} \times \STOP_{[t,T]}$ is an arbitrary Nash equilibrium of $\SDG_t(X,Y,Z)$, then $V^*$ is a submartingale on $[t,\htau_t]$ and a supermartingale on $[t,\hsi_t]$.

\noindent (v) For $\ep> 0$, if $(\tau^\ep_t, \si^\ep_t)\in \STOP_{[t,T]} \times \STOP_{[t,T]}$ is the pair of $\ep$-optimal strategies of $\SDG_t(X,Y,Z)$ defined by \eqref{eqev112}:
\begin{align}
\label{eqev205} \si^\ep_t=\inf\{u \geq t : Y_u \leq V^*_u+\ep\},\quad \tau^\ep_t=\inf\{u \geq t : X_u \geq V^*_u-\ep\},
\end{align}
then $V^*$ is a submartingale on $[t,\tau^\ep_t]$ and a supermartingale on $[t,\si^\ep_t]$.
\end{proposition}

\begin{proof} (i) Consider the right-continuous process defined by
\[
V^\si_t:=V^*_t \I_{\{t < \si\}} + Y_\si \I_{\{\si \leq t\}}.
\]
On the event $\{t<\si\}$, we have
\begin{gather}\label{eqev522}
V^\si_t=V^*_t=\essinf_{\si_t\in\Tstrat_{[t,T]}}\, \esssup_{\tau_t\in\Tstrat_{[t,T]}} \E_\P \big(R(\tau_t,\si_t)\,|\,  \Filt_{t}\big) \leq \esssup_{\tau_t\in\Tstrat_{[t,T]}} \E_\P \big(R(\tau_t,\si)\,|\,  \Filt_{t}\big) = Q^\si_t
\end{gather}
and on the event $\{\si\leq t\}$ we obtain
\begin{gather}\label{eqev523}
V^\si_t=Y_\si=R(T,\si) \leq \esssup_{\tau_t\in\Tstrat_{[t,T]}} \E_\P \big(R(\tau_t,\si)\,|\,  \Filt_{t}\big) = Q^\si_t.
\end{gather}
By combining \eqref{eqev522} and \eqref{eqev523}, we obtain $V^\si_t\leq Q^\si_t$. Applying Lemma \ref{lemev51} we have
\begin{gather*}
V^*_{\si\wedge\tau}=V^\si_{\si\wedge\tau} \leq Q^\si_\tau
\end{gather*}
as required.

\noindent (ii)  By using the optional sampling theorem on $Q$ and the $\ep$-strategy property of $\hsi_t$,
\begin{gather*}
\E_\P \big(P \,|\,  \Filt_{t}\big)
\leq \E_\P \big(Q^{\hsi_t}_{\tau_t} \,|\,  \Filt_{t}\big)
\leq Q^{\hsi_t}_t
= \esssup_{\rho\in\Tstrat_{[t,T]}} \E_\P \big(R(\rho,\hsi_t)\,|\,  \Filt_{t}\big)
\leq V^*_t + \ep,
\end{gather*}
as required.

\noindent (iii)  The lower bound of \eqref{eqev201} follows directly from parts (i) and (ii). The upper bound also follows by
the symmetry of the problem.

\noindent (iv)  To obtain the required result, it suffices to set $\ep=0$ in part (iii)

\noindent (v) Again we will only demonstrate the lower bound. By \eqref{eqev205}, $\si^\ep_t$ is increasing with respect to $\ep$. So for any $\delta\in[0,\ep]$, we have $\si^\delta_t$ being an $\delta$-optimal strategy with $\si^\ep_t \in \Tstrat_{[t,\si^\delta_t]}$. Hence by (iii),
\begin{gather*}
\E_\P \big( V^*_{\si^\ep_t}\,|\,  \Filt_{t}\big) \leq V^*_t +\delta.
\end{gather*}
Since this is true for all choice of $\delta\in[0,\ep]$, we must have $\E_\P \big( V^*_{\si^\ep_t}\,|\,  \Filt_{t}\big) \leq V^*_t$ as required.
\end{proof}

\begin{proposition}\label{propev25}
If $(\htau_t,\hsi_t)$ is a Nash equilibrium of $\SDG_t(X,Y,Z)$ then $(\htau_t\wedge \tau^0_t,\hsi_t\wedge \si^0_t)$ is also a Nash equilibrium, where $\tau^0_t, \si^0_t$ are defined by \eqref{eqev112}.
\end{proposition}

\begin{proof}
We will first show that $(\htau_t,\hsi_t\wedge \si^0_t)$ is a Nash equilibrium. It is sufficient to show that
\begin{gather}\label{eqev251}
\essinf_{\si_t \in\STOP_{[t,T]}} \E_\P \big( R(\htau_t,\si_t )\,|\,  \Filt_{t}\big) \geq V^*_t \geq \esssup_{\tau_t \in\STOP_{[t,T]}} \E_\P \big(
R(\tau_t ,\hsi_t\wedge \si^0_t)\,|\,  \Filt_{t}\big).
\end{gather}
The upper inequality is clear, since $(\htau_t,\hsi_t)$ is a Nash equilibrium. For the lower inequality, the key is to introduce $Q$, as defined in \eqref{eqev53}, and then apply Proposition \ref{propev20}(i).

 There are two cases to examine: \hfill \break
(a) on the event $\{\hsi_t \wedge \tau_t < \si^0_t\}$, we obtain
\begin{gather}\label{eqev252}
R(\tau_t ,\hsi_t\wedge \si^0_t) = R(\tau_t ,\hsi_t) = R(\tau_t \wedge \si^0_t,\hsi_t)\leq Q^{\hsi_t}_{\tau_t \wedge \si^0_t},
\end{gather}
where the last inequality follows from \eqref{eqev54}; \hfill \break
(b) on the event $\{\si^0_t \leq \hsi_t \wedge \tau_t \}$, we have that
\begin{gather}
\label{eqev253} R(\tau_t ,\hsi_t\wedge \si^0_t) = Y_{\si^0_t} \text{ or } Z_{\si^0_t} \leq Y_{\si^0_t}= V^*_{\si^0_t}\\
\label{eqev254} =V^*_{\tau_t \wedge \si^0_t \wedge \hsi_t} \leq Q^{\hsi_t}_{\tau_t \wedge \si^0_t}
\end{gather}
where the last equality of \eqref{eqev253} follows from Lemma \ref{lemev21}, and the last inequality of \eqref{eqev254} follows from \eqref{eqev55}.

Combining \eqref{eqev252} and \eqref{eqev254}, we conclude that in both cases
\begin{gather*} 
R(\tau_t ,\hsi_t\wedge \si^0_t) \leq Q^{\hsi_t}_{\tau_t \wedge \si^0_t}.
\end{gather*}
Now apply Proposition \ref{propev20}(ii), setting $P=R(\tau_t ,\hsi_t\wedge \si^0_t)$ and $\ep=0$,
\begin{align*}
\E_\P \big(R(\tau_t ,\hsi_t\wedge \si^0_t)\,|\,  \Filt_{t}\big) \leq V^*_t.
\end{align*}
This establishes \eqref{eqev251} and thus $(\htau_t,\hsi_t\wedge \si^0_t)$ is a Nash equilibrium. Finally, using similar arguments to replace $\htau_t$ by $\htau_t\wedge \tau^0_t$, we obtain the required result.
\end{proof}

\subsection{General Dynkin Game}\label{sec3.22}

The goal of this section is to study the \emph{general Dynkin game} $\GDG_t(X,Y,Z)$ with the payoff
\begin{gather}\label{eqex01}
R(\tau , \si ) = \I_{\{ \tau < \si \}}\, X_{\tau }  +  \I_{\{ \si  < \tau \}}\,
Y_{\si }  +  \I_{\{ \si  = \tau \}}\, Z_{\si }.
\end{gather}
Hence we no longer postulate that $ X \leq Z \leq Y$.
Similarly as in Section \ref{sec3.12}, our
goal here is to find the necessary and sufficient conditions for the following property:
\begin{gather}
\textit{For all $t\in[0,T]$ and $\ep>0$,
the Dynkin game $\GDG_t(X,Y,Z)$ has $\ep$-optimal strategies.}\label{eqex001}
\end{gather}
Furthermore, we would also like to explore the necessary and sufficient conditions for the following property:
\begin{gather}
\textit{For all $t\in[0,T]$,
the Dynkin game $\GDG_t(X,Y,Z)$ has a Nash equilibrium.}\label{eqex002}
\end{gather}

Motivated by the discrete-time case examined in Subsection \ref{sec3.12}, we begin by defining $L:=Z\wedge X$ and $U:=Z\vee Y$. It is clear that $L$ and $U$ are c\`adl\`ag processes satisfying the usual integrability condition and $L\leq Z\leq U$. Again it makes sense to consider the Dynkin game $\SDG (L,U,Z)$ associated with the payoff
\begin{gather}\label{eqex02}
\wt R(\tau , \si ) = \I_{\{ \tau < \si \}}\, L_{\tau }  +  \I_{\{ \si  < \tau \}}\,
U_{\si }  +  \I_{\{ \si  = \tau \}}\, Z_{\si }.
\end{gather}
In light of Theorem \ref{thmev11}, we introduce the following notation.

\begin{definition}\label{defex07}
(i) The process $V$ is given by
\begin{align} \label{eqex071}
V_t = \essinf_{\si_t \in\STOP_{[t,T]}} \,
\esssup_{\tau_t \in\STOP_{[t,T]}} \E_\P \big(\wt R(\tau_t ,\si_t )\,|\,
\Filt_{t}\big)
=  \esssup_{\tau_t \in\STOP_{[t,T]}} \, \essinf_{\si_t \in\STOP_{[t,T]}} \E_\P \big(\wt R(\tau_t ,\si_t )\,|\,
\Filt_{t}\big)
\end{align}
where $\wt R$ is given by \eqref{eqex02}.

\noindent (ii)  For each $t\in[0,T]$ and $\ep\geq 0$, define the $\FF$-stopping times $\si^\ep_t, \tau^\ep_t \in \Tstrat_{[t,T]}$ by
\begin{gather}\label{eqex072}
\si^\ep_t:=\inf\{u \geq t : U_u \leq V_u+\ep\},\quad \tau^\ep_t:=\inf\{u \geq t : L_u \geq V_u-\ep\}.
\end{gather}
\end{definition}

We again note that $\lim_{\ep\to 0} \si^\ep_t\leq \si^0_t$ and $\lim_{\ep\to 0} \tau^\ep_t\leq \tau^0_t$, but equality may fail to hold.

By Theorem \ref{thmev11}, $V$ is the value process of $\SDG (L,U,Z)$ and for $\ep>0$, $(\si^\ep_t, \tau^\ep_t) \in\Tstrat_{[t,T]}\times\Tstrat_{[t,T]}$ is a pair of $\ep$-optimal strategies for $\SDG_t(L,U,Z)$. The goal is to show that $V$ and $(\si^\ep_t, \tau^\ep_t)$ are also the value process and $\ep$-optimal strategies, respectively, of the Dynkin game $\GDG (X,Y,Z)$.

We begin by observing that an analogue of Lemma \ref{lemeo12} can be readily applied to the continuous-time case.

\begin{lemma}\label{lemex12} For $t\in[0,T]$, the following properties are valid.

\noindent (i)  For any fixed $\tau_t,\si_t\in\Tstrat_{[t,T]}$, there exist $\htau_t, \hsi_t \in\Tstrat_{[t,T]}$ such that
\begin{gather}\label{eqex121}
R(\htau_t , \si_t ) \geq \wt R(\tau_t , \si_t ) \geq R(\tau_t , \hsi_t ).
\end{gather}

\noindent (ii) The value $V_t$ of $\SDG_t(L,U,Z)$ lies between the minimax and the maximin values of the game $\GDG_t(X,Y,Z)$. In other words,
\begin{gather}\label{eqex122}
\essinf_{\si_t \in\STOP_{[t,T]}} \, \esssup_{\tau_t \in\STOP_{[t,T]}}
\E_\P \big(R(\tau_t , \si_t )\,|\,  \Filt_{t}\big) \geq V_t \geq
\esssup_{\tau_t \in\STOP_{[t,T]}} \, \essinf_{\si_t \in\STOP_{[t,T]}}
\E_\P \big(R(\tau_t , \si_t )\,|\,  \Filt_{t}\big).
\end{gather}

\noindent (iii) If the $\GDG_t(X,Y,Z)$ has a value then it equals to $V_t$.
\end{lemma}

\begin{proof} (i) The proof is identical to the proof of Lemma \ref{lemeo12}. We will only prove the upper inequality in \eqref{eqex121}, as the lower inequalities follows by symmetry. To choose a stopping time $\htau$ such that $R(\htau_t , \si_t ) \geq \wt R(\tau_t , \si_t ) $, we first compare $R(\tau_t , \si_t )$ and $\wt R(\tau_t , \si_t ) $.  On the following events, $R(\tau_t , \si_t ) \geq\wt R(\tau_t , \si_t )$ is automatically satisfied
\begin{align*}
{\{\tau_t = \si_t\}},&\quad R(\tau_t , \si_t ) = Z_{\tau_t}=\wt R(\tau_t , \si_t ),\\
{\{\tau_t < \si_t\}},&\quad R(\tau_t , \si_t ) = X_{\tau_t}\geq L_{\tau_t} =\wt R(\tau_t , \si_t ),\\
\{ \si_t < \tau_t ,\, Y_{\si_t}\geq Z_{\si_t}\},&\quad R(\tau_t , \si_t )=Y_{\si_t}=U_{\si_t} = \wt R(\tau_t , \si_t ).
\end{align*}
The problem arises on the event $\{\si_t < \tau_t ,\, Z_{\si_t}>Y_{\si_t}\}$, since then
\[
R(\tau_t , \si_t )=Y_{\si_t} < U_{\si_t} = \wt R(\tau_t , \si_t ).
\]
Let us modify $\tau$ by setting
\begin{gather}
\htau = \si_t \I_{\{\si_t < \tau_t, Z_{\si_t}>Y_{\si_t}\}} + \tau_t \big(1- \I_{\{\si_t < \tau_t, Z_{\si_t}>Y_{\si_t}\}}\big).
\end{gather}
Then $\htau$ is indeed an $\FF$-stopping time, since the event $\{\si_t < \tau_t, Z_{\si_t}>Y_{\si_t}\}$ belongs to $\Filt_{\si_t\wedge \tau_t}$. Furthermore, on the event $\{\si_t < \tau_t, Z_{\si_t}>Y_{\si_t}\}$ we have that
\[
R(\htau_t , \si_t ) = R(\si_t , \si_t ) = Z_{\si_t} = U_{\si_t} =\wt R(\tau_t , \si_t )
\]
and thus for the stopping time $\htau $ the left-hand side inequality in \eqref{eqex121} is satisfied.

\noindent (ii) Again, we only show the upper inequality of \eqref{eqex122}. By Theorem \ref{thmev11}, $V_t$ is the value of the game $\SDG_t(L,U,Z)$. For any $\ep>0$, $(\tau^\ep_t, \si^\ep_t)$ (see Definition \ref{defex07}(ii)) is a pair of $\ep$-optimal strategy for $\SDG_t(L,U,Z)$. Hence we have, for any $\si_t\in\Tstrat_{[t,T]}$,
\begin{gather*}
\E_\P \big(\wt R(\tau^\ep_t , \si_t )\,|\,  \Filt_{t}\big) \geq  V_t - \ep.
\end{gather*}
By part (i), there exists $\htau_t\in\Tstrat_{[t,T]}$ such that $R(\htau_t , \si_t )\geq \wt R(\tau^\ep_t , \si_t )$. Consequently,
\begin{gather}\label{lemex1221}
\esssup_{\tau_t \in\STOP_{[t,T]}}
\E_\P \big(R(\tau_t , \si_t )\,|\,  \Filt_{t}\big) \geq
\E_\P \big(R(\htau_t , \si_t )\,|\,  \Filt_{t}\big) \geq
\E_\P \big(\wt R(\tau^\ep_t , \si_t )\,|\,  \Filt_{t}\big) \geq V_t - \ep.
\end{gather}
Since \eqref{lemex1221} holds for all $\si_t\in\Tstrat_{[t,T]}$ and $\ep>0$, we must have
\begin{gather*}
\essinf_{\si_t \in\STOP_{[t,T]}} \, \esssup_{\tau_t \in\STOP_{[t,T]}}
\E_\P \big(R(\tau_t , \si_t )\,|\,  \Filt_{t}\big) \geq V_t,
\end{gather*}
as required.


\noindent (iii) The proof is the same as in Lemma \ref{lemex12}. By the definition of the value (see Definition \ref{defaa03}), if there exists a value $V^*_t$ for the game $\GDG_t(X,Y,Z)$, then it must satisfy
\begin{gather}\label{lemex1231}
V^*_t=\essinf_{\si_t \in\STOP_{[t,T]}} \, \esssup_{\tau_t \in\STOP_{[t,T]}}
\E_\P \big(R(\tau_t , \si_t )\,|\,  \Filt_{t}\big) =
\esssup_{\tau_t \in\STOP_{[t,T]}} \, \essinf_{\si_t \in\STOP_{[t,T]}}
\E_\P \big(R(\tau_t , \si_t )\,|\,  \Filt_{t}\big).
\end{gather}
In view of part (ii), we conclude that the equality $V^*_t=V_t$ necessarily holds.
\end{proof}

Based on the intuition of the discrete case (see Subsection \ref{sec3.12}), we begin with the following condition, with the aim of achieving \eqref{eqex001} and \eqref{eqex002}.

\begin{assumption}\label{assex40}
Let $X, Y$ and $Z$ be $\FF$-adapted integrable, c\`adl\`ag processes and let the associated process $V$ be given as in Definition \ref{defex07}(i). We postulate that the processes $X, Y$ and $V$ satisfy, for all $t\in[0,T]$,
\begin{gather}\label{assex41}
X_t\wedge Y_t \leq V_t \leq X_t\vee Y_t .
\end{gather}
\end{assumption}

\subsubsection{Sufficiency of Assumption \ref{assex40}}

\begin{proposition} \label{propex50}
For all $t\in[0,T]$, $\ep\geq 0$, let $V_t, \si^\ep_t$ and $\tau^\ep_t$ be defined as in Definition \ref{defex07}.
Under Assumption \ref{assex40}, we have the following:

\noindent (i)  For some $\ep\geq 0$, if $\si_t\in\Tstrat_{[t,T]}$ satisfies $\si_t\leq \si^\ep_t$, then for all $\tau_t\in\Tstrat_{[t,T]}$,
\begin{gather}\label{eqex501}
R(\tau_t ,\si_{t}) \leq Q^{\si^{\ep}_t}_{\tau_t}.
\end{gather}
where $Q^{\si_t}$ is defined by
\begin{gather*}
Q^{\si_t}_u:= \esssup_{\rho\in\Tstrat_{[u,T]}} \E_\P \big(\wt R(\rho,\si_t)\,|\,  \Filt_{u}\big),\quad u\in[t,T].
\end{gather*}

\noindent (ii) The process $V$ is the value process of $\GDG (X,Y,Z)$. For all $\ep>0$, the stopping times $\si^\ep_t, \tau^\ep_t$ are $\ep$-optimal strategies of $\GDG_t(X,Y,Z)$, satisfying
\begin{gather}\label{eqex502}
\essinf_{\si_t \in\STOP_{[t,T]}} \E_\P \big( R(\tau^{\ep}_{t},\si_t )\,|\,  \Filt_{t}\big)+\ep \geq V_t \geq \esssup_{\tau_t \in\STOP_{[t,T]}} \E_\P \big(
R(\tau_t ,\si^{\ep}_{t})\,|\,  \Filt_{t}\big)-\ep.
\end{gather}

\noindent (iii) If $(\tau^*_t,\si^*_t)$ is an arbitrary Nash equilibrium of $\SDG_t(L,U,Z)$, then $(\tau^*_t\wedge \tau^0_t,\si^*_t\wedge \si^0_t)$ is a Nash equilibrium of $\GDG_t(X,Y,Z)$.
\end{proposition}

\begin{proof} (i) We will make use of \eqref{eqev54} and \eqref{eqev55}, that is,
\[
\wt R(\tau_{t},\si_t ) \vee V_{\tau_t\wedge \si_t} \leq Q^{\si_t}_{\tau_t}
\]
There are a few cases to check: \hfill \break
(a) On the event $\{\si_t = \tau_t\}$,
\begin{gather}\label{eqex51}
R(\tau_{t},\si_t ) = Z_{\si_t} = \wt R(\tau_{t},\si_t ) \leq Q^{\si_t}_{\tau_t}.
\end{gather}
(b) On the event $\{\si_t < \tau_t\}$,
\begin{gather}\label{eqex52}
R(\tau_{t},\si_t ) = Y_{\si_t} \leq U_{\si_t} = \wt R(\tau_{t},\si_t ) \leq Q^{\si_t}_{\tau_t}.
\end{gather}
(c) On the event $\{\tau_t<\si_t \}$, certainly $\tau_t<\si_t\leq \si^\ep_t$. From the definition of $\si^\ep_t$ in Definition \ref{defex07}(ii), we must have
\begin{gather}\label{eqex521}
V_{\tau_t} < U_{\tau_t}.
\end{gather}
We now consider the following subcases:

\noindent (c.1) If $Y_{\tau_t} \geq Z_{\tau_t}$, then by \eqref{eqex521} $Y_{\tau_t} = U_{\tau_t} > V_{\tau_t}$. Since Assumption \ref{assex40} requires $V$ to lie between $X$ and $Y$, we must have
\begin{gather}\label{eqex53}
R(\tau_{t},\si_t ) = X_{\tau_t} \leq V_{\tau_t} \leq Q^{\si_t}_{\tau_t}.
\end{gather}
(c.2) If $Y_{\tau_t} < Z_{\tau_t}$, then by \eqref{eqex521} $Z_{\tau_t} = U_{\tau_t} > V_{\tau_t}$. Now by Lemma \ref{lemev21}(i), $V_{\tau_t} \geq L_{\tau_t} = Z_{\tau_t} \wedge X_{\tau_t}$. Hence we must have
\begin{gather}\label{eqex54}
R(\tau_{t},\si_t ) = X_{\tau_t} = L_{\tau_t} \leq V_{\tau_t} \leq Q^{\si_t}_{\tau_t}.
\end{gather}

In view of \eqref{eqex51}, \eqref{eqex52}, \eqref{eqex53} and \eqref{eqex54}, we conclude that $R(\tau_{t},\si_t ) \leq Q^{\si_t}_{\tau_t}$ for all cases, establishing \eqref{eqex501}.

\noindent (ii)  By Proposition \ref{propep05} and Lemma \ref{lemex12}(iii), it is sufficient to establish \eqref{eqex502}, or
\[
\essinf_{\si_t \in\STOP_{[t,T]}} \E_\P \big( R(\tau^{\ep}_{t},\si_t )\,|\,  \Filt_{t}\big)+\ep \geq V_t \geq \esssup_{\tau_t \in\STOP_{[t,T]}} \E_\P \big(
R(\tau_t ,\si^{\ep}_{t})\,|\,  \Filt_{t}\big)-\ep.
\]
We will only establish the lower bound, since the upper bound follows by symmetry.
From part (i), we know that $R(\tau_t ,\si^{\ep}_{t})\leq Q^{\si^\ep_t}_{\tau_t}$ for all $\tau_t\in\Tstrat_{[t,T]}$.
Since $\si^\ep_t$ is an $\ep$-optimal strategy, we can apply Proposition \ref{propev20}(i).
By setting $P= R(\tau_{t},\si^{\ep}_t)$, we have, for all $\tau_t\in\Tstrat_{[t,T]}$,
\begin{gather*}
\E_\P \big(R(\tau_{t},\si^{\ep}_t) \,|\,  \Filt_{t}\big) \leq V^*_t + \ep.
\end{gather*}
Hence
\begin{gather*}
\esssup_{\tau_t\in\Tstrat_{[t,T]}} \E_\P \big(R(\tau_{t},\si^{\ep}_t) \,|\,  \Filt_{t}\big) \leq V^*_t + \ep ,
\end{gather*}
as required.

\noindent (iii) By Proposition \ref{propev25}, $(\tau^*_t\wedge \tau^0_t,\si^*_t\wedge \si^0_t)$ is also a Nash equilibrium of $\SDG_t(L,U,Z)$, satisfying $\tau^*_t\wedge \tau^0_t\leq \tau^0_t$ and $\si^*_t\wedge \si^0_t \leq \si^0_t$. Since a Nash equilibrium is also a pair of $0$-optimal strategies, we can simply use the same argument as before, but with $\ep=0$.
\end{proof}

In general, not all $\ep$-optimal strategies (resp. Nash equilibria) of $\SDG_t(L,U,Z)$ are necessarily $\ep$-optimal strategies (resp. Nash equilibria) of $\GDG_t(X,Y,Z)$. Proposition \ref{propex50} only applies to $\ep$-optimal strategies (resp. Nash equilibria) stopping no later than $\tau^\ep_t$ and $\si^\ep_t$ (resp. $\tau^0_t$ and $\si^0_t$).

\subsubsection{Necessity of Assumption \ref{assex40}}

\begin{proposition}\label{propey10}
Suppose that Assumption \ref{assex40} is violated at time $t\in[0,T]$, that is, almost surely
\begin{gather}\label{eqey20}
\{V_t < X_t \wedge Y_t\} \cup \{V_t > X_t \vee Y_t\} \neq \emptyset .
\end{gather}
Then there exists $\ep>0$ such that the Dynkin game $\GDG_t(X,Y,Z)$ does not have $\ep$-optimal strategies. In particular, $\GDG_t(X,Y,Z)$ has no Nash equilibrium.
\end{proposition}

\begin{proof}
Since $X_T=Y_T=Z_T=V_T$ then manifestly \eqref{eqey20} cannot occur when $t=T$. Assume, for the sake of contradiction, that \eqref{eqey20} holds for some $t<T$ and there exists a pair of $\ep$-optimal strategies $(\tau^\ep_t,\si^\ep_t)$ for all $\ep>0$. Then, by Proposition \ref{propep05} and Lemma \ref{lemex12}(iii), $V_t$ must be the value of $\GDG_t(X,Y,Z)$.

Assume now that either $\P(V_t < X_t \wedge Y_t)>0$ or $\P(V_t > X_t \vee Y_t)>0$. First, consider the event $\{V_t < X_t \wedge Y_t\}$. Then there exists $\ep>0$ such that $\P(V_t + \ep < X_t \wedge Y_t)>0$. On that event, let us consider
\[
\tau'_t= t\I_{\{\si^\ep_t>t\}}+ T\I_{\{\si^\ep_t=t\}}.
\]
Then $R(\tau'_t,\si^\ep_t)$ is either $X_t$ or $Y_t$. But then
\[
\esssup_{\tau_t \in\STOP_{[t,T]}} \E_\P \big(R(\tau_t ,\si^{\ep}_{t})\,|\,  \Filt_{t}\big) \geq \E_\P \big(R(\tau'_t ,\si^{\ep}_{t})\,|\,  \Filt_{t}\big)\geq X_t\wedge Y_t > V_t + \ep,
\]
contradicting the $\ep$-optimal property of $\si^{\ep}_{t}$.
The same argument can be applied to the event $\{V_t > X_t \vee Y_t\}$. Hence there exists $\ep>0$ such that $\GDG_t(X,Y,Z)$ does not have $\ep$-optimal strategies.
\end{proof}

\begin{proposition}\label{propey15}
Under Assumption \ref{assex40}, if $(\tau^*_t,\si^*_t)$ is an arbitrary Nash equilibrium of $\GDG_t(X,Y,Z)$, then it is also a Nash equilibrium of $\SDG_t(L,U,Z)$.
\end{proposition}
\begin{proof}
We want to prove that, for all $\si_t,\tau_t\in\Tstrat_{[t,T]}$,
\begin{gather}\label{propey151}
\wt R(\tau^*_t,\si_t) \geq \wt R(\tau^*_t,\si^*_t) \geq \wt R(\tau_t,\si^*_t).
\end{gather}
By Lemma \ref{lemex12}(i), there exists $\htau_t\in\Tstrat_{[t,T]}$ such that
\begin{gather*} 
R(\htau_t,\si^*_t) \geq \wt R(\tau_t,\si^*_t).
\end{gather*}
Since $(\tau^*_t,\si^*_t)$ is a Nash equilibrium of $\GDG_t(X,Y,Z)$,
\begin{gather*} 
V_t = R(\tau^*_t,\si^*_t) \geq R(\htau_t,\si^*_t) \geq \wt R(\tau_t,\si^*_t).
\end{gather*}
Hence the lower bound of \eqref{propey151} is established. The upper bound can be proven similarly. Therefore, $(\tau^*_t,\si^*_t)$ is a Nash equilibrium of $\SDG_t(L,U,Z)$.
\end{proof}

To summarise the necessity and sufficiency results of this section, we now combine Theorem \ref{thmev11} with Propositions \ref{propex50}, \ref{propey10} and \ref{propey15}.

\begin{theorem}\label{thmez10}
Suppose $X, Y, Z$ are integrable c\`adl\`ag progressive processes satisfying $X_T=Y_T=Z_T$ and let $L=X\wedge Z$ and $U=Y\vee Z$. Consider the family of Dynkin games $\GDG (X,Y,Z)$ associated with the payoff
\begin{gather*} 
R(\tau , \si ) = \I_{\{ \tau < \si \}}\, X_{\tau }  +  \I_{\{ \si  < \tau \}}\,
Y_{\si }  +  \I_{\{ \si  = \tau \}}\, Z_{\si }.
\end{gather*}

\noindent (i) The Dynkin game $\GDG_t(X,Y,Z)$ has a value and a pair of $\ep$-optimal strategies for all $t\in[0,T]$ and $\ep>0$ if and only if Assumption \ref{assex40} holds. In particular, the unique value process $V^*$ is given by
\begin{gather*} 
V^*_t = \essinf_{\si_t \in\STOP_{[t,T]}} \,
\esssup_{\tau_t \in\STOP_{[t,T]}} \E_\P \big(R(\tau_t ,\si_t )\,|\,
\Filt_{t}\big)
=  \esssup_{\tau_t \in\STOP_{[t,T]}} \, \essinf_{\si_t \in\STOP_{[t,T]}} \E_\P \big(R(\tau_t ,\si_t )\,|\,
\Filt_{t}\big)
\end{gather*}
and a pair of $\ep$-optimal strategies $(\tau^\ep_t,\si^\ep_t)$ is given by
\begin{gather*} 
\si^\ep_t:=\inf\{u \geq t : U_u \leq V_u+\ep\},\quad \tau^\ep_t:=\inf\{u \geq t : L_u \geq V_u-\ep\}.
\end{gather*}

\noindent (ii) The Dynkin game $\GDG_t(X,Y,Z)$ has a Nash equilibrium for all $t\in[0,T]$ if and only if Assumption \ref{assex40} holds and the Dynkin game $\SDG_t(L,U,Z)$ has a Nash equilibrium for all $t\in[0,T]$. If we further assume that $L$ and $-U$ only have positive jumps, then $\GDG_t(X,Y,Z)$ has a Nash equilibrium $(\tau^*_t,\si^*_t)$ given by
\begin{gather*} 
\si^*_t=\lim_{\ep\to 0} \si^\ep_t,\quad \tau^*_t=\lim_{\ep\to 0} \tau^\ep_t.
\end{gather*}

\noindent (iii) Fix $t\in[0,T]$. If Assumption \ref{assex40} holds, then $\GDG_t(X,Y,Z)$ has a Nash equilibrium if and only if $\SDG_t(L,U,Z)$ has a Nash equilibrium.
\end{theorem}

\def\skipold{

\subsection{Zero-Sum Dynkin Games: General Case (old)}

The goal of this section is to study a more general class of zero-sum Dynkin games, associating with the payoff
\begin{gather}\label{eqff13}
R(\tau , \si ) = \I_{\{ \tau < \si \}}\, X_{\tau }  +  \I_{\{ \si  < \tau \}}\,
Y_{\si }  +  \I_{\{ \si  = \tau \}}\, Z_{\si }.
\end{gather}
As mentioned in (??),  it is common to make the assumption $X\leq Z\leq Y$. We will attempt to construct a more general set of assumptions allowing, in principle, for any relative positions of $X, Y$ and $Z$.

We begin by introducing the following notation. We adopt the usual convention that $\inf \emptyset = T$.

\begin{definition} \label{assff14}
\begin{enumerate}[(i)]
\item\label{assff141} We denote by $\rho$ the first moment when $X$ crosses the random level $Y$, that is,
\begin{gather}
\rho:=\inf \big\{ t\in [0, T]\,|\, X_t \geq Y_t \big\}.
\end{gather}
\item\label{assff142} Given $\tau,\si \in \STOP_{[0,T]}$, we set
\begin{align}
\htau &:= \inf \big\{ t\in [\tau, T]\,|\, Z_t \geq X_t \wedge Y_t \big\},\\
\tsi &:= \inf \big\{ t\in [\si, T]\,|\, Z_t \leq X_t \vee Y_t \big\}.
\end{align}
\item\label{assff143} We denote by $Z'_\rho$ the $\Filt_\rho$-measurable random variable given as
\begin{gather}\label{eqff18}
Z'_\rho = \big(Z_\rho \vee (X_\rho \wedge Y_\rho)\big) \wedge (X_\rho \vee Y_\rho) = \big(Z_\rho \wedge (X_\rho \vee Y_\rho)  \big)\vee (X_\rho \wedge Y_\rho).
\end{gather}
\end{enumerate}
\end{definition}

Since $X, Y$ and $Z$ are $\FF$-adapted and right-continuous, the stopping times $\rho, \htau$ and $\tsi$ are well defined. Furthermore, since $X_\rho \wedge Y_\rho \leq X_\rho \vee Y_\rho$, the random variable $Z'_\rho$ is also well defined.

\begin{lemma}\label{lemff200} The following properties hold:
\begin{enumerate}[(P1)]
\item \label{lemff201} $X_\rho \wedge Y_\rho \leq Z'_\rho \leq X_\rho \vee Y_\rho $;
\item \label{lemff203} $Y_t > X_t$ on the event $\{t<\rho\}$;
\item \label{lemff205} $X_\rho \geq Y_\rho$ on the event $\{\rho<T\}$.
\end{enumerate}
\end{lemma}

\begin{proof} The first property is immediate from \eqref{eqff18}.
Properties (P\ref{lemff203}) and (P\ref{lemff205})  follow from the definition of $\rho$
and the right-continuity of processes $X$ and $Y$.
\end{proof}

We conjecture that the following conditions are sufficient for the existence of a Nash equilibrium.

\begin{assumption}\label{assff15}
Let the processes $X, Y$ and $Z$ satisfy the following conditions:
\begin{enumerate}[(i)]
\item\label{assff151} $X_T\wedge Y_T \leq Z_T \leq X_T\vee Y_T$;
\item\label{assff152} On $[0,\rho]$, $X$ and $-Y$ are left upper semi-continuous (only have positive jumps); furthermore, $\limsup_{t \uparrow \rho} X_t \leq Y_\rho$ and $\liminf_{t \uparrow \rho} Y_t \geq X_\rho$;
\item\label{assff153} For any $\tau,\si \in \STOP_{[0,\rho]}$, the process $X\wedge Y$ is non-decreasing on $[\tau,\htau]$ and the process  $X\vee Y$ is non-increasing on $[\si,\tsi]$.
\end{enumerate}
\end{assumption}

\begin{remark}\label{remff161}
Assumption \ref{assff15} \eqref{assff153} is equivalent to the following condition:
for each $\omega\in\Omega$, on the event $\{t \leq \rho(\omega)\}$,
\begin{alignat*}{4}
Z_t &< X_t \wedge Y_t & \quad&\implies\quad& X_s \wedge Y_s &\leq X_u\wedge Y_u, &\quad \forall\  s,u \in [t,\hat t],\, s\leq u,\\
Z_t &> X_t \vee Y_t & \quad&\implies\quad&  X_s \vee Y_s &\geq  X_u \vee Y_u, &\quad \forall\  s,u \in [t,\hat t],\, s\leq u.
\end{alignat*}
(more explanation)
\end{remark}

Under Assumption \ref{assff15}, we have the following additional properties (P\ref{lemff204})--(P\ref{lemff207}).

\begin{lemma}\label{lemff20} Let $\tau,\si \in\Tstrat_{[0,T]}$ be arbitrary $\FF$-stopping times, and let $\rho,\htau,\tsi$ and $Z'_\rho$ be as in Definition \ref{assff14}. Under Assumption \ref{assff15}, we have the following properties:
\begin{enumerate}[(P1)] \setcounter{enumi}{3}
\item\label{lemff204} On the event $\{\tau<\rho\}$, we have $X_t \geq X_t \wedge Y_t \geq X_\tau$ for $\tau\leq t\leq \htau$;
\item\label{lemff206} On the event $\{\tau\leq \rho<\htau\}$, we have $X_t \wedge Y_t\geq Z'_\rho > Z_\rho$ for $\rho\leq t\leq \htau$;
\\ On the event $\{\si \leq \rho<\tsi\}$, we have $X_t \vee Y_t\leq Z'_\rho < Z_\rho$ for $\rho\leq t\leq \hat \tsi$;
\item\label{lemff207} $Z_{\htau} \geq X_{\htau} \wedge Y_{\htau}$ and $Z_{\tsi} \leq X_{\tsi} \vee Y_{\tsi}$.
\end{enumerate}
\end{lemma}

\begin{proof}
\begin{enumerate}
\item[(P\ref{lemff204})] By Assumption \ref{assff15} \eqref{assff153}, we have $X_u \wedge Y_u \geq X_t \wedge Y_t$. Then from (P\ref{lemff203}), $X_t \wedge Y_t=X_t$.
On the event $\{t<\rho\}$, we have $X_u \geq X_u \wedge Y_u \geq X_t$ for $t\leq u\leq \hat t$;
\item[(P\ref{lemff206})] We will only prove the first statement; the second can be shown using similar arguments.
Since $\tau'\in\Tstrat_{[0,\rho]}$, we have $\rho \in [\tau',\htau')$. From the definition of $\htau'$ \eqref{eqff191} and Assumption \ref{assff15} \eqref{assff153}, we have $Z_\rho< X_\rho \wedge Y_\rho \leq X_u \wedge Y_u$ for $\rho\leq u\leq \htau'$. Finally, by the definition of $Z'_\rho$, we have that $Z'_\rho = X_\rho \wedge Y_\rho$, completing the statement.
\item[(P\ref{lemff207})] This follows from the definitions of $\htau$ and $\tsi$, Assumption \ref{assff15} \eqref{assff151}, and the right-continuity of $X, Y$ and $Z$.
\end{enumerate}
\end{proof}

Before dealing with the Dynkin game $\Game$ with payoff given by \eqref{eqff13}, we will examine Dynkin games $\Game'$ and $\Game''$ with the following modified payoff functions.
\begin{align}
R'(\tau , \si ) &= \I_{\{ \tau < \si \wedge \rho \}}\, X_{\tau }  +  \I_{\{ \si  < \tau \wedge \rho \}}\,
Y_{\si } + \I_{\{ \tau = \si < \rho \}}\, X_{\tau } +  \I_{\{ \rho \leq \tau \wedge \si \}}\, Z'_{\rho },\label{eqff20}\\
R''(\tau , \si ) &= \I_{\{ \tau < \si \wedge \rho \}}\, X_{\tau }  +  \I_{\{ \si  < \tau \wedge \rho \}}\,
Y_{\si } + \I_{\{ \tau = \si < \rho \}}\, Z_{\tau } +  \I_{\{ \rho \leq \tau \wedge \si \}}\, Z'_{\rho },\label{eqff201}
\end{align}

\begin{remark}\label{remff162}
In both $\Game'$ and $\Game''$, the payoff for stopping after $\rho$ is constant; in other words, for any $\tau,\si \in\Tstrat_{[0,T]}$,
\begin{gather*}
R'(\tau , \si ) = R'(\tau , \si \wedge \rho ) = R'(\tau \wedge \rho , \si )  = R'(\tau \wedge \rho , \si \wedge \rho ),\\
R''(\tau , \si ) = R''(\tau , \si \wedge \rho ) = R''(\tau \wedge \rho , \si )  = R''(\tau \wedge \rho , \si \wedge \rho ).
\end{gather*}
Hence when analysing $\Game'$ and $\Game''$, we can assume, without loss of generality, that $\tau,\si \in\Tstrat_{[0,\rho]}$.
\end{remark}

To see that $\Game'$ belongs to the class described by \eqref{eqff11}, it suffices to note that:
\begin{itemize}
\item From (P\ref{lemff203}), $X_t\leq Y_t$ for $t<\rho$;
\item From (P\ref{lemff201}) and Assumption \ref{assff15} \eqref{assff152},
\[
\limsup_{t \uparrow \rho} X_t \leq X_\rho \wedge Y_\rho \leq Z'_\rho \leq X_\rho \vee Y_\rho \leq \liminf_{t \uparrow \rho} Y_t .
\]
\end{itemize}
By Theorem \ref{thmff11}, $\Game'$ has a Nash equilibrium $(\tau',\si')$ with $\tau',\si'\in \Tstrat_{[0,\rho]}$. This will be used as a starting point to construct $(\htau', \tsi')$, where $\htau', \tsi'$ are taken as in Definition \ref{assff14} \eqref{assff142}.
Out ultimate goal is to prove that the pair $(\htau', \tsi')$ is a Nash equilibrium for all three games $\Game', \Game''$ and $\Game$.

\begin{proposition}\label{lemff21}
\begin{enumerate}[(i)]
\item\label{lemff211} For any $\tau, \si\in\STOP_{[0,\rho]}$, we have
$R'(\htau , \si ) \geq R'(\tau , \si ) \geq R'(\tau , \tsi).$
\item\label{lemff212} The pair of stopping times $(\htau', \tsi')$
is a Nash equilibrium of $\Game'$.
\item\label{lemff213} On the event $\{\htau' = \tsi'\}$, we have $\{\htau' = \tsi' \geq \rho\}$ almost surely.
\item\label{lemff214} The pair of stopping times $(\htau', \tsi')$ is a Nash equilibrium of $\Game''$.
\item\label{lemff215} $R(\htau' , \tsi' ) = R''(\htau' , \tsi' )$.
\end{enumerate}
\end{proposition}

\begin{proof} (i) Since $\tau, \si\in\STOP_{[0,\rho]}$, we have $\tau \leq \htau \wedge \rho$ and $\si\leq \tsi  \wedge \rho$. We will now consider several cases.
\begin{itemize}
\item On the event $\{\htau \leq \si \} \cap \{\htau < \rho \}$, we have $\{\tau \leq \si \} \cap \{\tau < \rho \}$ and hence by (P\ref{lemff204}),
\[
R'(\htau , \si ) = X_{\htau } \geq X_{\tau } = R'(\tau , \si ).
\]
\item On the event $\{ \si < \htau \wedge \rho \}$, we have two cases
\begin{itemize}
\item If $\tau \leq \si < \htau \wedge \rho$ then, by (P\ref{lemff203}) and  (P\ref{lemff204}),
\[
R'(\htau , \si ) = Y_{\si } > X_{\si } \geq X_{\tau}= R'(\tau , \si ).
\]
\item If $\si < \tau \leq \htau \wedge \rho$ then
\[
R'(\htau , \si ) = Y_{\si }= R'(\tau , \si ).
\]
\end{itemize}
\item On the event $\{\rho \leq  \htau \wedge \si  \}$, we have two cases
\begin{itemize}
\item If $\tau < \rho \leq  \htau \wedge \si$ then, by (P\ref{lemff201}) and (P\ref{lemff204}),
\[
R'(\htau , \si ) = Z'_{\rho } \geq X_{\rho } \wedge Y_{\rho} \geq X_{\tau}= R'(\tau , \si ).
\]
\item If $\rho \leq  \tau \wedge \si$ then
\[
R'(\htau , \si ) = Z'_{\rho } = R'(\tau , \si ).
\]
\end{itemize}
\end{itemize}
In all cases we have $R'(\htau , \si )\geq R'(\tau , \si )$, proving the upper inequality. The lower inequality can be established by similar arguments.

\noindent (ii)  Recall that $(\tau',\si')$ is a Nash equilibrium with $\tau',\si' \in\STOP_{[0,\rho]}$. Applying \eqref{lemff211} yields
\begin{gather}
\E\big(R'(\htau' , \tsi' )\big) \geq \E\big(R'(\tau' , \tsi' )\big) \geq \E\big(R'(\tau' , \si')\big)
\geq \E\big(R'(\htau' , \si' )\big) \geq \E\big(R'(\htau' , \tsi' )\big). \label{eqff212}
\end{gather}
Hence all terms of \eqref{eqff212} are equal and $R'(\htau' , \tsi' ) = R'(\tau' , \si')$. Now, for any $\tau,\si\in\STOP_{[0,T]}$,
\begin{align}
\E\big(R'(\htau' , \si )\big) \geq \E\big(R'(\tau' , \si )\big) \geq \E\big(R'(\tau' , \si')\big) &= \E\big(R'(\htau' , \tsi')\big),\\
\E\big(R'(\tau , \tsi' )\big) \leq \E\big(R'(\tau , \si' )\big) \leq \E\big(R'(\tau' , \si')\big) &= \E\big(R'(\htau' , \tsi')\big).
\end{align}
Therefore, the pair $(\htau' , \tsi')$ is a Nash equilibrium of $\Game'$.

\noindent (iii) Assume the contrary, that is, $\P\{\htau' = \tsi' < \rho\}>0$. Consider the $\FF$-stopping time
\[
\bar \tau = \rho \I_{\{\htau' = \tsi' < \rho\}} + \htau' (1-  \I_{\{\htau' = \tsi' < \rho\}}).
\]
(explain why $\bar \tau$ is a stopping time). By (P\ref{lemff203}), the inequality $Y_t>X_t$ holds for $t<\rho$
and thus we obtain
\begin{align*}
\E\big(R'(\bar \tau, \tsi')\big) &=  \E\big(R'(\htau', \tsi') (1- \I_{\{\htau' = \tsi' < \rho\}}) + Y_{\htau'} \I_{\{\htau' = \tsi' < \rho\}}\big)\\
&> \E\big(R'(\htau', \tsi') (1- \I_{\{\htau' = \tsi' < \rho\}}) + X_{\htau'} \I_{\{\htau' = \tsi' < \rho\}}\big)\\
& = \E\big(R'(\htau', \tsi')\big).
\end{align*}
This contradicts the property established in part \eqref{lemff212} that $(\htau', \tsi')$ is a Nash equilibrium.

\noindent (iv) It suffices to show that
\begin{gather}\label{eqff214}
\E\big(R''(\htau' , \si )\big) \geq \E\big(R''(\htau' , \tsi' )\big) \geq \E\big(R''(\tau , \tsi' )\big),\quad \forall\,\si,\tau\in\Tstrat_{[0,\rho]}.
\end{gather}
From \eqref{eqff20} and \eqref{eqff201},
\begin{gather*}
R'(\tau , \si) \neq R''(\tau , \si ) \quad\implies\quad  \tau = \si < \rho.
\end{gather*}
From \eqref{lemff213}, we get $\P(\{\htau' = \tsi' < \rho\})=0$, and thus
\begin{gather*}
\E \big(R'(\htau' , \tsi' )\big) = \E \big(R''(\htau' , \tsi' )\big).
\end{gather*}
To establish the upper inequality of \eqref{eqff214}, we combine (P\ref{lemff203}) and (P\ref{lemff207}) to obtain $Z_{\htau'} \geq X_{\htau'}$ on the event $\{\htau' < \rho\}$. Consequently, for any $\si\in\Tstrat_{[0,\rho]}$,
\begin{align*}
R''(\htau' , \si ) &= Z_{\htau'} \I_{\{\htau' = \si < \rho\}} + R''(\htau' , \si ) (1-\I_{\{\htau' = \si < \rho\}})\\
&\geq X_{\htau'} \I_{\{\htau' = \si < \rho\}} + R''(\htau' , \si ) (1-\I_{\{\htau' = \si < \rho\}})\\
&= R'(\htau', \rho ) \I_{\{\htau' = \si < \rho\}} + R'(\htau' , \si ) (1-\I_{\{\htau' = \si < \rho\}})\\
&= R'(\htau', \bar \si ) \\
&\geq R'(\htau', \tsi'),
\end{align*}
where $\bar\si = \rho \I_{\{\htau' = \si < \rho\}} + \si (1-\I_{\{\htau' = \si < \rho\}})$. Consequently,
\[
\E\big( R''(\htau' , \si )\big) \geq \E\big( R'(\htau', \tsi')\big) = \E\big( R''(\htau', \tsi')\big),
\]
as required. The lower inequality of \eqref{eqff214} can be shown using similar arguments.

\noindent (v) On the event $\{\htau'\wedge\tsi' < \rho\}$, the result follows directly from the definition of $R$ and $R''$.
On the event $\{\htau'\wedge\tsi' > \rho\}$, by (P\ref{lemff206}), we have $Z_\rho > Z'_\rho > Z_\rho$,
which is a contradiction.

 We are thus left with the cases $\{\rho = \htau' < \tsi'\}$ and $\{\rho = \tsi' < \htau'\}$. For the case of $\{\rho = \htau' < \tsi'\}$, we have $\rho<T$ and thus, by (P\ref{lemff205}), $X_\rho \geq Y_\rho$. Therefore, by (P\ref{lemff201}),
\[
R(\htau' , \tsi' ) = X_\rho = X_\rho \vee Y_\rho \geq Z'_\rho = R''(\htau' , \tsi' )
\]
as required. The case of $\{\rho = \tsi' < \htau'\}$ can be dealt with similarly.
\end{proof}

 We are in a position to prove the main result of this section. Proposition \ref{lemff21}, \eqref{lemff212} and \eqref{lemff214} showed that $(\htau', \tsi')$ is a Nash equilibrium of $\Game'$ and $\Game''$. The next result shows that it is also a Nash equilibrium of $\Game$.

\begin{proposition}\label{propff25}
The pair of stopping times $(\htau', \tsi')$ is a Nash equilibrium of the Dynkin game $\Game$, that is,
\begin{gather}\label{eqff251}
\E\big(R(\htau' , \si )\big) \geq \E\big(R(\htau' , \tsi' )\big) \geq \E\big(R(\tau , \tsi' )\big),\quad \forall\,\si,\tau\in\Tstrat_{[0,T]}.
\end{gather}
\end{proposition}

\begin{proof}
To prove the upper inequality of \eqref{eqff251}, we will first establish the following inequality
\begin{gather*} 
R(\htau' , \si ) \geq R''(\htau' , \si ),\quad \forall\,\si\in\Tstrat_{[0,T]}.
\end{gather*}
On the event $\{\htau' \wedge \si < \rho\}$ and by definition of $R$ and $R''$ we have $R(\htau' , \si ) = R''(\htau' , \si )$. For $\{\rho \leq \htau' \wedge \si \}$, it suffices to show $R(\htau' , \si ) \geq R''(\htau' , \si ) = Z'_\rho$. There are three cases to consider.

\begin{itemize}
\item On the event $\{ \rho = \htau'\leq \si\}$, there are two sub-cases.
\begin{itemize}
\item If $\rho =\htau'=\si$ then, by (P\ref{lemff207}), we have $Z_{\rho} \geq X_{\rho} \wedge Y_{\rho}$. Hence by the definition of $Z'_\rho$ (see Definition \ref{assff14} \eqref{assff143}),
\[
R(\htau' , \si ) = Z_\rho \geq  Z_{\rho} \wedge (X_{\rho} \vee Y_{\rho}) = Z'_\rho.
\]
\item If $\rho =\htau' < \si$ then $\rho<T$. By (P\ref{lemff201}) and (P\ref{lemff203}),
\[
R(\htau' , \si ) = X_\rho = X_\rho \vee Y_\rho \geq Z'_\rho.
\]
\end{itemize}
\item On the event $\{ \rho < \htau' \leq \si\}$, by (P\ref{lemff206}), we have $X_{\htau'} \wedge Y_{\htau'} \geq Z'_\rho$. Again, there are two sub-cases.
\begin{itemize}
\item If $\rho < \htau' = \si$ then, by (P\ref{lemff207}), we obtain
\[
R(\htau' , \si ) = Z_{\htau'} \geq X_{\htau'} \wedge Y_{\htau'} \geq Z'_\rho.
\]
\item If $\rho < \htau' < \si$ then we have
\[
R(\htau' , \si ) = X_{\htau'} \geq X_{\htau'} \wedge Y_{\htau'} \geq Z'_\rho.
\]
\end{itemize}
\item On the event $\{ \rho < \si < \htau'\}$, by (P\ref{lemff206}), we have $X_{\si} \wedge Y_{\si} \geq Z'_\rho$. Consequently,
\[
R(\htau' , \si ) = Y_{\si} \geq X_{\si} \wedge Y_{\si} \geq Z'_\rho.
\]
\end{itemize}
In all cases, we have $R(\htau' , \si )\geq R''(\htau' , \si )$. Therefore, by parts \eqref{lemff214} and \eqref{lemff215}
 in Proposition \ref{lemff21}, we obtain
\begin{gather*}
\E\big(R(\htau' , \si )\big) \geq \E\big(R''(\htau' , \si )\big) \geq \E\big(R''(\htau' , \tsi' )\big) = \E\big(R(\htau' , \tsi' )\big),\quad \forall\,\si\in\Tstrat_{[0,\rho]}.
\end{gather*}
This establishes the upper inequality of \eqref{eqff251}. The lower inequality can be proven using similar arguments.
\end{proof}

%

\subsection{Non-Zero-Sum Dynkin Games}

As in Section \ref{sec3.13}, we consider the following general payoff functions
\begin{gather*} 
R^1(\tau , \si ) = \I_{\{ \tau < \si \}}\, X^1_{\tau }  +  \I_{\{ \si  < \tau \}}\,
Y^1_{\si }  +  \I_{\{ \si  = \tau \}}\, Z^1_{\tau },\\
R^2(\tau , \si ) = \I_{\{ \si < \tau \}}\, X^2_{\si }  +  \I_{\{ \tau  < \si \}}\,
Y^2_{\tau }  +  \I_{\{ \si  = \tau \}}\, Z^2_{\si }.
\end{gather*}
For convenience write $X=(X^1,X^2)$, $Y=(Y^1,Y^2)$, $Z=(Z^1,Z^2)$, $R=(R^1,R^2)$.

The following special case was studied by Hamad\`ene and Zhang \cite{hamadene2010continuous}
\begin{gather*} 
R^1(\tau , \si ) = \I_{\{ \tau \leq \si \}}\, X^1_{\tau }  +  \I_{\{ \si  < \tau \}}\, Y^1_{\si } ,\\
R^2(\tau , \si ) = \I_{\{ \si \leq \tau \}}\, X^2_{\si }  +  \I_{\{ \tau  < \si \}}\, Y^2_{\tau } .
\end{gather*}

\begin{theorem}\label{thmff11a}
Assume that the following conditions hold:
\begin{enumerate}[(i)]
\item $X^i \leq Y^i$ for $i=1,2$;
\item the processes $X^1$ and $X^2$ are left upper semi-continuous (only have positive jumps),
\item for any $\tau \in \STOP_{[0,T]}$, $ \P ( \{ X^1_t < Y^1_t \} \setminus  \{ X^2_t < Y^2_t \} )=0$.
\end{enumerate}
Then the Dynkin game associated with the payoff $R= (R^1,R^2)$ has a Nash equilibrium.
\end{theorem}

Define the following regions in $\R^2$:
\begin{gather*}
\mathbb{X}^0_t=\big\{(x,y)\in\R^2 : x\leq X^1_t, y\leq X^2_t \big\},\qquad \mathbb{Y}^0_t=\big\{(x,y)\in\R^2 : x\leq Y^1_t, y\leq Y^2_t \big\},\\
\mathbb{X}^1_t=\big\{(x,y)\in\R^2 : x<X^1_t, y>X^2_t \big\},\qquad \mathbb{Y}^2_t=\big\{(x,y)\in\R^2 : x<Y^1_t, y>Y^2_t \big\},\\
\mathbb{X}^2_t=\big\{(x,y)\in\R^2 : x>X^1_t, y<X^2_t \big\},\qquad \mathbb{Y}^1_t=\big\{(x,y)\in\R^2 : x>Y^1_t, y<Y^2_t \big\},\\
\mathbb{X}^3_t=\big\{(x,y)\in\R^2 : x\geq X^1_t, y\geq X^2_t \big\},\qquad \mathbb{Y}^3_t=\big\{(x,y)\in\R^2 : x\geq Y^1_t, y\geq Y^2_t \big\}.
\end{gather*}

Introduce the following stopping times:

\begin{definition} \label{assfj04}
In the following definitions, we adopt the convention of setting $\inf \emptyset = T$.
\begin{enumerate}[(i)]
\item\label{assfj041} Denote by $\rho$ the first instance of $X^i$ increasing above $Y^i$ for $i=1,2$, that is
\begin{gather*}
\rho:=\inf \big\{ t\in [0, T]\,|\, X^1_t \geq Y^1_t \big\} \wedge \inf \big\{ t\in [0, T]\,|\, X^2_t \geq Y^2_t \big\}.
\end{gather*}
\item\label{assfj042} Given $\tau,\si \in \STOP_{[0,T]}$, define
\begin{align*}
\htau &:= \inf \big\{ t\in [\tau, T]\,|\, Z_t \notin \mathbb{Y}^1 \big\},\\
\tsi &:= \inf \big\{ t\in [\si, T]\,|\, Z_t \notin \mathbb{Y}^2 \big\}.
\end{align*}
\end{enumerate}
\end{definition}

\begin{conjecture}\label{propfj25}
The following conditions are sufficient for the existence of a Nash equilibrium.
\begin{enumerate}[(i)]
\item\label{assfj051} $Z_T\in\mathbb{Y}^0_T \cup \mathbb{Y}^3_T$;
\item\label{assfj052} On $[0,\rho]$, $X^1$ and $X^2$ are left upper semi-continuous (only have positive jumps); furthermore, at $\rho$, $\limsup_{t \uparrow \rho} X^i_t \leq Y^i_\rho$ for $i=1,2$;
\item\label{assfj054} For any $\tau \in \STOP_{[0,\rho]}$, on the event $\{\tau<\htau\}$, $X^1_\tau \leq Y^1_\tau$ and $X^1$ is non-decreasing on $[\tau,\htau]$;
For any $\si \in \STOP_{[0,\rho]}$, on the event $\{\si<\tsi\}$, $X^1_\si \leq Y^1_\si$ and $X^1$ is non-decreasing on $[\si,\tsi]$.
\item\label{assfj053} On the event $\{X_\rho\in\mathbb{Y}^1_\rho\}$, there exists $\bar \rho>\rho$ such that $X^2_t \geq X^2_\rho$ for all $t\in[\rho,\bar \rho]$;
On the event $\{X_\rho\in\mathbb{Y}^2_\rho\}$, there exists $\tilde \rho>\rho$ such that $X^1_t \geq X^1_\rho$ for all $t\in[\rho,\tilde \rho]$;
\end{enumerate}
\end{conjecture}

}

\bibliographystyle{acm}
\bibliography{DynkinGamesRefs}

\end{document}